\documentclass[11pt]{article}
\usepackage[reqno]{amsmath}
\usepackage{amssymb}
\usepackage{wasysym}
\usepackage{mathrsfs}
\usepackage{fancyhdr}
\usepackage{verbatim}
\usepackage{color}
\usepackage{amsrefs}
\usepackage{amsthm}
\usepackage{dsfont}
\usepackage{graphicx}
\usepackage[all]{xy}
\usepackage[colorlinks=true]{hyperref}
\usepackage{mathtools}

\DeclarePairedDelimiter{\floor}{\lfloor}{\rfloor}
\newcommand{\de}{\partial}

\newcommand{\take}{\,\backslash\,}

\newcommand{\wt}{\widetilde}
\newcommand{\Ideal}[1]{\left\langle#1\right\rangle}

\newcommand{\R}{\mathbb{R}}
\DeclareMathOperator{\dive}{div}

\DeclareMathOperator{\Span}{span}

\DeclareMathOperator{\range}{range}

\DeclareMathOperator{\End}{End}

\DeclareMathOperator{\dom}{dom}
\DeclareMathOperator{\supp}{supp}

\DeclareMathOperator{\spin}{spin}
\DeclareMathOperator{\Char}{Char}

\newcommand{\Th}{^\text{\textnormal{th}}}
\newcommand{\ol}{\overline}

\newcommand{\wh}{\widehat}

\newcommand{\Cl}{\mathbb{C}\text{\textnormal{l}}}
\newcommand{\C}{\mathbb{C}}

\newcommand{\bs}{\boldsymbol}
\newcommand{\D}{\mathcal{D}}
\newcommand{\A}{\mathcal{A}}
\newcommand{\h}{\mathcal{H}}

\newcommand{\E}{\mathcal{E}}

\newcommand{\Z}{\mathbb{Z}}

\newcommand{\bC}{\mathfrak{C}}
\newcommand{\bc}{\mathfrak{c}}

\newcommand{\T}{\mathbb{T}}
\newtheorem{thm}{Theorem}
\newtheorem{lemma}[thm]{Lemma}
\newtheorem{prop}[thm]{Proposition}
\newtheorem{cor}[thm]{Corollary}
\theoremstyle{definition}
\newtheorem{defn}[thm]{Definition}

\newtheorem*{rem}{Remark}

\advance\topmargin by -\headheight
\advance\topmargin by -\headsep
\evensidemargin=\oddsidemargin
\makeatletter
\def\section{\@startsection{section}{1}{\z@}{-3.5ex plus -1ex minus
  -.2ex}{2.3ex plus .2ex}{\large\bf}}
\def\subsection{\@startsection{subsection}{2}{\z@}{-3.25ex plus -1ex
  minus -.2ex}{1.5ex plus .2ex}{\normalsize\bf}}
\makeatother

\author{Iain Forsyth$^{\dag\S}$, Adam Rennie$^\S$\thanks{email: 
\texttt{iain.forsyth@anu.edu.au}, \texttt{renniea@uow.edu.au}
},
\\[3pt]
\dag Mathematical Sciences Institute, Australian National University,\\  Canberra, Australia
\\[3pt]
\S School of Mathematics and Applied Statistics, University of Wollongong,\\
Wollongong, Australia\\
}

\title{Factorisation of equivariant spectral triples in unbounded $KK$-theory.}

\date{}

\begin{document}
\maketitle
\parindent=0.0in
\parskip=4pt
\vspace{-16pt}
\begin{abstract}  
We provide sufficient conditions to factorise an equivariant spectral triple as a Kasparov
product of unbounded classes constructed from the group action on the algebra and
from the fixed point spectral triple. Our results are for the action of compact abelian Lie groups, and
we demonstrate them with examples from manifolds and $\theta$-deformations. In particular
we show that equivariant Dirac-type spectral triples on the total space of a torus principal bundle
always factorise. We also present
an example that shows what goes wrong in the absence of our sufficient conditions
(and how we get around it for this example).
\end{abstract} 
AMS 2010: 19K35, 46L87.
\section{Introduction}
In this paper we provide sufficient conditions to factorise a 
$G$-equivariant spectral triple $(\A,\h,\D)$, for $G$ compact abelian, 
as a Kasparov product of a `fixed point' spectral
triple and a Kasparov module constructed solely from the action of the group on the algebra.
More precisely, given our sufficient conditions, we find unbounded cycles representing classes in $KK_G^{\dim G}(A,A^G)$ and $KK_G^{j+\dim G}(A^G,\C)$, with $A$ the norm completion of $\A$,
such that  the  Kasparov product of these classes 
$$
KK_G^{\dim G}(A,A^G)\times KK_G^{j+\dim G}(A^G,\C)\rightarrow KK_G^j(A,\C)
$$
recovers the class of $(\A,\h,\D)$ in $KK^j(A,\C)$.

In order to  define the Kasparov module 
with class in $KK_G^{\dim G}(A,A^G)$, we require 
that the action of $G$ on $A$ satisfies the spectral subspace assumption of \cite{CNNR}. 
To define the unbounded Kasparov module with class in 
$KK_G^{j+\dim G}(A^G,\C)$, we need a Clifford action 
$$
\eta:\Gamma(\Cl(G))^{G}\cong\Cl_{\dim G}\rightarrow B(\h)
$$
satisfying a few compatibility conditions. Finally,  the product of these classes
represents the class of $(\A,\h,\D)$, 
provided that one positivity constraint is satisfied: this constraint
arises from Kucerovsky's criteria \cite{Kucerovsky}. 

Our factorisation results show that the class of our equivariant spectral triple is the 
product of classes with unbounded representatives, which are defined in terms
of the original spectral triple subject to some geometric constraints. The constructive 
approach to the Kasparov product, \cites{BMS,KaadLesch,Mesland,MR}, seeks to
construct a spectral triple from unbounded representatives of composable $KK$-classes.
Having obtained a factorisation, say,
$$
[(\A,\h,\D)]=[(\A',E_{A^G},\D_1)]\wh{\otimes}_{A^G}[(\A^G,\h_2,\D_2)]
$$
it is natural to ask whether the constructive product of $(\A',E_{A^G},\D_1)$ and 
$(\A^G,\h_2,\D_2)$ makes sense and recovers the original triple $(\A,\h,\D)$. We examine this 
question for an equivariant Dirac-type spectral triple $(C^\infty(M),L^2(S),\D)$ on
a compact Riemannian manifold with a free isometric torus action, where we
show that factorisation holds in our sense. In this special case, 
we show that the constructive method produces a spectral triple $(C^\infty(M),L^2(S),T)$
whose
$KK$-class is the same as that of $(C^\infty(M),L^2(S),\D)$. The operator $T$ 
is a self-adjoint elliptic first order differential operator, but the difference $\D-T$ is
typically unbounded. If each orbit in $M$
is an isometrically embedded copy of $\T^n$, we find that $\D-T$ is bounded. Thus we see evidence
in these examples that the constructive product is sensitive to metric data.

Factorisation of circle-equivariant spectral triples has 
also been studied in \cite{BMS}, \cites{DS2013,DSZ2014} and the 
Ph.D. thesis of A. Zucca, \cite{Zucca}. The last three of these works study such factorisations under the condition 
of ``fibres of constant length'', a condition which is also satisfied in the
examples studied in \cite{BMS}. Such a condition appears in 
Corollary \ref{cor:connectionfactorisation}, and corresponds to the isometric embedding
of orbits (up to a constant multiple).

Finally, we consider in detail the factorisation of the 
Dirac operator over the 2-sphere, for rotation by the circle. 
In this case, the circle action is not free and 
factorisation for $C(S^2)$ is not possible, but we 
show that factorisation is nevertheless possible if 
one restricts to the $C^*$-algebra of continuous functions vanishing at the poles.

\textbf{Acknowledgements.} Both authors were supported by the Australian Research Council, 
and the authors thank the University of Wollongong, the Australian National University and
the Hausdorff Institute for Mathematics
for hospitality.
We would also like to thank Simon Brain, Magnus Goffeng and Walter van Suijlekom 
for useful discussions.
\section{The construction of the unbounded $KK$-cycles.}
\begin{defn}\label{defn:kasmodule}
Let $A$ and $B$ be $\Z_2$-graded 
$C^*$-algebras carrying respective actions 
$\alpha$ and $\beta$ by a compact group $G$. 
An \emph{unbounded equivariant Kasparov 
$A$-$B$-module} $(\A,E_B,\D)$ consists of an invariant 
dense sub-$\ast$-algebra $\A\subset A$, a 
countably generated $\Z_2$-graded 
right Hilbert $B$-module $E$ with a homomorphism 
$V$ from $G$ into the invertible degree zero bounded 
linear (not necessarily adjointable) operators on $E$, 
a $\Z_2$-graded $\ast$-homomorphism 
$\phi:A\rightarrow\End_B(E)$, and an odd, 
self-adjoint, regular operator $\D:\dom(\D)\subset E\rightarrow E$ such that:\\
(1) $V_g(\phi(a)eb)=\phi(\alpha_g(a))V_g(e)\beta_g(b)$ and
$(V_ge|V_gf)_B=\beta_g((e|f)_B)$ for all $g\in G$, $a\in A$, $e\in E$ and $b\in B$;\\
(2) $\phi(a)\cdot\dom(\D)\subset\dom(\D)$, and the graded commutator
$[\D,\phi(a)]_\pm$ is bounded for all $a\in\A$;\\
(3) $\phi(a)(1+\D^2)^{-1/2}$ is a compact endomorphism for all $a\in\A$;\\
(4) $V_g\cdot\dom(\D)\subset\dom(\D)$, and $[\D,V_g]=0$.
\end{defn}
\begin{rem}
We normally suppress the notation $\phi$. 
The unbounded Kasparov module $(\A,E_B,\D)$ defines a class in the abelian group $KK_G(A,B)$,
\cite{BaajJulg}.
\end{rem}
\begin{rem}
We will only employ unbounded equivariant Kasparov $A$-$B$-modules 
for which the action of $G$ on $B$ is trivial. Then  for all $g\in G$, $V_g$ is adjointable with adjoint $V_g^*=V_{g^{-1}}$.
\end{rem}
\begin{defn}
Let $A$ be a $\Z_2$-graded $C^*$-algebra with an 
action by a compact group $G$. An \emph{even equivariant spectral triple} $(\A,\h,\D)$ 
for $A$ is an unbounded equivariant Kasparov $A$-$\C$-module. 
If $A$ is trivially $\Z_2$-graded, then one can also 
define an \emph{odd equivariant spectral triple} $(\A,\h,\D)$, 
which has the same definition, except that $\h^1=\{0\}$ and $\D$ need not be odd. 
\end{defn}
Throughout this section, $G$ is a compact abelian 
Lie group, equipped with the normalised Haar measure, 
and $(\A,\h,\D)$ is an even $G$-equivariant spectral triple for a 
$\Z_2$-graded separable $C^*$-algebra $A$ 
carrying an action $\alpha$ by $G$. (The case that the spectral triple is odd is considered later.)

There are some differences between the cases of $G$ 
even dimensional and $G$ odd dimensional. We introduce 
the following notation so that we may handle both cases simultaneously.
\begin{defn}
\label{defn:bc}
Let $\Cl_1$ be the Clifford algebra generated 
by a self-adjoint unitary $c$, which is $\Z_2$-graded by
$$
\Cl_1^j=\Span\{c^j\},\quad j\in\Z_2.
$$
We denote by $\bC$ the $\Z_2$-graded $C^*$-algebra
$$
\bC=\left\{\begin{array}{ll}\C&\text{if }G\text{ is even dimensional}\\
\Cl_1&\text{if }G\text{ is odd dimensional.}\end{array}\right.
$$
We also denote by $\bc$ the generator of $\bC$; i.e.
$$
\bc=\left\{\begin{array}{ll}1&\text{if }G\text{ is even dimensional}\\
c&\text{if }G\text{ is odd dimensional.}\end{array}\right.
$$
\end{defn}
We will construct three unbounded $KK$-cycles. The first cycle 
(referred to as the left-hand module), is constructed using the spin 
Dirac operator over $G$, and defines a class in $KK_G(A,A^G\wh{\otimes}\bC)$. 
The second cycle, which we call the middle module, represents a class in 
$KK_G(A^G\wh{\otimes}\bC,A^G\wh{\otimes}\Gamma(\Cl(G))^G)$. 
The module is simply the Morita equivalence between 
$A^G\wh{\otimes}\bC$ and $A^G\wh{\otimes}\Gamma(\Cl(G))^G\cong A^G\wh{\otimes}\Cl_n$, 
and so contains no homological information. The third cycle (the right-hand module) 
is constructed by restricting the spectral triple to a spectral subspace of $\h$, 
and adding a representation of $\Gamma(\Cl(G))^G$, so that it defines a 
class in $KK_G(A^G\wh{\otimes}\Gamma(\Cl(G))^G,\C)$.
\subsection{The left-hand module.}
Let $\Char(G)$ be the characters of $G$, which is 
the set of smooth homomorphisms $\chi:G\rightarrow U(1)$. 
Since $G$ is abelian, the characters form a group under multiplication. 
For each $\chi\in\Char(G)$, let
$$
A_\chi=\{a\in A:\alpha_g(a)=\chi(g)a\}
$$
be the \emph{spectral subspace} of $A$ associated 
with the character $\chi$. Note that $\bigoplus_{\chi\in\Char(G)}A_\chi$ is dense in $A$.
For each $\chi\in\Char(G)$, define $\Phi_\chi:A\rightarrow A$ by
$$
\Phi_\chi(a)=\int_G\chi^{-1}(g)\alpha_g(a)\,dg.
$$
Each $\Phi_\chi$ is a continuous idempotent with $\range\Phi_\chi=A_\chi$.
\begin{defn}
The action of $G$ on $A$ is said to satisfy the \emph{spectral subspace assumption} 
(SSA) if the norm closure $\ol{A_\chi A_\chi^*}$ is a complemented ideal in the 
fixed point algebra $A^G$ for each $\chi\in\Char(G)$.
\end{defn}
We define an $A^G$-valued inner product on $A$ by
$$
(a|b)_{A^G}:=\Phi_1(a^*b)=\int_G\alpha_g(a^*b)\,dg.
$$
With this inner product, $A$ is a right pre-Hilbert $A^G$-module. 
Hence the completion of $A$ with respect to $(\cdot|\cdot)_{A^G}$ 
is a right Hilbert $A^G$-module, which we denote by $X$.
The $\Z_2$-grading of $A$ defines a $\Z_2$-grading of 
$X$, which makes $X$ into a $\Z_2$-graded right Hilbert 
$A^G$-module. The action of $G$ on $A$ extends to a unitary action 
$\alpha:G\rightarrow\End_{A^G}(X)$.\begin{rem}
Let $\chi\in\Char(G)$, and let $a,b\in A_\chi$. Then 
$a^*b\in A^G$, so $(a|b)_{A^G}=a^*b$. Hence $A_\chi$ is closed in $X$, and so
$$
X_\chi:=\{x\in X:\alpha_g(x)=\chi(g)x\}=A_\chi.
$$
\end{rem}
The following is a more general version of \cite{PR}*{Lemma 4.2} 
or \cite{CNNR}*{Lemma 2.4}. The result there is for the case 
$G=\mathbb{T}$, but the proof is much the same as in the general case.
\begin{lemma}
\label{lem:innerproductsum}
For each $\chi\in\Char(G)$, the map $\Phi_\chi:A\rightarrow A$ 
extends to an adjointable projection $\Phi_\chi:X\rightarrow X$ 
with range $A_\chi$. Moreover,
$$
(x|y)_{A^G}=\sum_{\chi\in\Char(G)}\Phi_\chi(x)^*\Phi_\chi(y)
$$
for all $x,y\in X$, and the sum $\sum_{\chi\in\Char(G)}\Phi_\chi$ 
converges strictly to the identity on $X$.
\end{lemma}
Let $\bs{\$}_G$ be the trivial flat complex spinor bundle over $G$, with 
Dirac operator $\D_G$. The left multiplication of $G$
on itself lifts to a strongly continuous unitary representation 
$V$ on $L^2(\bs{\$}_G)$ which makes 
$(C^\infty(G),L^2(\bs{\$}_G),\D_G)$ into a $G$-equivariant spectral triple,
which is even if and only if $\dim G$ is even, \cite{Slebarski1985}. Then 
$(C^\infty(G), (L^2(\bs{\$}_G)\wh{\otimes}\bC)_\bC,\D_G\wh{\otimes}\bc)$ 
is a $G$-equivariant unbounded Kasparov $C(G)$-$\bC$-module 
for $G$ either even or odd dimensional. 
\begin{defn}
\label{defn:pchi}
Let $X\wh{\otimes}(L^2(\bs{\$}_G)\wh{\otimes}\bC)$ be the external tensor product of $X$ and $L^2(\bs{\$}_G)\wh{\otimes}\bC$, which is a $\Z_2$-graded right Hilbert $A^G\wh{\otimes}\bC$-module. Let $E_1$ be the invariant submodule of $X\wh{\otimes}(L^2(\bs{\$}_G)\wh{\otimes}\bC)$ under the diagonal action $g\cdot(x\wh{\otimes}(s\wh{\otimes}z))=\alpha_g(x)\wh{\otimes}(V_gs\wh{\otimes}z)$. Let $V_1$ be the homomorphism from $G$ into the unitaries of $E_1$ defined by
$$
V_{1,g}(x\wh{\otimes}(s\wh{\otimes}z))=\alpha_g(x)\wh{\otimes}(s\wh{\otimes}z).
$$
For each $\chi\in\Char(G)$, let $p_\chi'\in B(L^2(\bs{\$}_G))$ be the orthogonal projection onto
$$
L^2(\bs{\$}_G)_\chi=\{s\in L^2(\bs{\$}_G):V_g(s)=\chi(g)s\},
$$
and define $p_\chi\in\End_{\bC}(L^2(\bs{\$}_G)\wh{\otimes}\bC)$ by
$
p_\chi(s\wh{\otimes}z)=p_\chi's\wh{\otimes}z.
$
\end{defn}
The following result is elementary, but will be quite useful in later calculations.
\begin{lemma}\label{lem:prodesp}
For elements of homogeneous degree, the $A^G\wh{\otimes}\bC$-valued inner product on $E_1$ can be expressed (for $x_1,\,x_2\in X$ and $s_1,\,s_2\in L^2(\bs{\$}_G)\wh{\otimes}\bC$) as 
\begin{align*}
(x_1\wh{\otimes}s_1|x_2\wh{\otimes}s_2)_{A^G\wh{\otimes}\bC}=(-1)^{\deg s_1\cdot(\deg x_1+\deg x_2)}\sum_{\chi\in\Char(G)}\Phi_\chi(x_1)^*\Phi_\chi(x_2)\wh{\otimes}(p_{\chi^{-1}} s_1|p_{\chi^{-1}} s_2)_{\bC}.
\end{align*}
\end{lemma}
\begin{prop}
Define an action of $\bigoplus_{\chi\in\Char(G)}A_\chi$ on $E_1$ by
$$
\sum_{\chi\in\Char(G)}a_\chi\cdot(x\wh{\otimes} s):=\sum_{\chi\in\Char(G)}a_\chi x\wh{\otimes}\chi s,
\quad\mbox{for}\ \  \sum a_\chi\in \bigoplus_{\chi\in\Char(G)}A_\chi,\ \ x\wh{\otimes}s\in E_1.
$$
This action extends to a $\Z_2$-graded $\ast$-homomorphism 
$\phi:A\rightarrow\End_{A^G\wh{\otimes}\bC}(E_1)$ satisfying
$$
V_{1,g}(\phi(a)e)=\phi(\alpha_g(a))V_{1,g}(e),\qquad a\in A,\ e\in E_1.
$$
\end{prop}
\begin{proof}
Suppose $a_\chi\in A_\chi$ and $x=\sum_{\nu\in\Char(G)}x_\nu\in X$, where $x_\nu\in A_\nu$ for all $\nu\in\Char(G)$. Then
$$\|a_\chi x\|^2=\sum_{\phi\in\Char(G)}\|a_\chi x_\nu\|^2\leq\|a_\chi\|^2\|x\|^2$$
by Lemma \ref{lem:innerproductsum}, so $a_\chi x$ is a well-defined element of $x$.

Since $\alpha_g(a_\chi^*)=\alpha_g(a_\chi)^*=\ol{\chi(g)}a_\chi^*=\chi^{-1}(g)a_\chi^*$, it follows that $a_\chi^*\in A_{\chi^{-1}}$. Hence if $a_\chi\in A_\chi$ and $x_i\wh{\otimes} s_i\in E_1$, $i=1,2$, each of homogeneous degree, then
\begin{align*}
\big(x_1\wh{\otimes} s_1\big|a_\chi\cdot(x_2\wh{\otimes} s_2)\big)_{A^G\wh{\otimes}\bC}&=\big(x_1\wh{\otimes} s_1\big|a_\chi x_2\wh{\otimes} \chi s_2\big)_{A^G\wh{\otimes}\bC}\\
&=(-1)^{\deg s_1\cdot(\deg x_1+\deg a_\chi+\deg x_2)}(x_1|a_\chi x_2)_{A^G}\wh{\otimes}(s_1|\chi s_2)_{\bC}\\
&=(-1)^{\deg s_1\cdot(\deg x_1+\deg a_\chi+\deg x_2)}(a_\chi^*x_1|x_2)_{A^G}\wh{\otimes}(\chi^{-1} s_1|s_2)_{\bC}\\
&=\big(a_\chi^*x_1\wh{\otimes}\chi^{-1} s_1\big|x_2\wh{\otimes}s_2\big)_{A^G\wh{\otimes}\bC}=\big(a_\chi^*\cdot(x_1\wh{\otimes} s_1)\big|x_2\wh{\otimes} s_2\big)_{A^G\wh{\otimes}\bC}.
\end{align*}
So the action of $\bigoplus_{\chi}A_\chi$ on $E_1$ defines a 
$\ast$-homomorphism 
$\bigoplus_{\chi}A_\chi\rightarrow\End_{A^G\wh{\otimes}\bC}(E_1)$, 
which extends to a $\ast$-homomorphism 
$\phi:A\rightarrow\End_{A^G\wh{\otimes}\bC}(E_1)$. 
That $\phi$ is $\Z_2$-graded and equivariant is obvious.
\end{proof}
\begin{defn}
Let $\D_G:\dom(\D_G)\subset L^2(\bs{\$}_G)\rightarrow L^2(\bs{\$}_G)$ 
be the spin Dirac operator on $G$, and let $\bc$ be the generator of 
$\bC$. Define a closed operator 
$\D_1:\dom(\D_1)\subset E_1\rightarrow E_1$ 
initially on the linear span of elements of the form 
$x\wh{\otimes}(s\wh{\otimes} z)$, where $x\in X$, 
$s\in\dom(\D_G)$ and $z\in\bC$ are of homogeneous degree, by
$$
\D_1(x\wh{\otimes}(s\wh{\otimes}z)):=(-1)^{\deg x}x\wh{\otimes}(\D_Gs\wh{\otimes}\bc z),
$$
and then take the operator closure. Since $\D_G$ is equivariant, this is well-defined.
\end{defn}
\begin{prop}
\label{prop:SSA3}
The triple $(\oplus_\chi A_\chi,(E_1)_{A^G\wh{\otimes}\bC},\D_1)$ is an 
unbounded equivariant Kasparov $A$-$A^G\wh{\otimes}\bC$-module 
if and only if the action of $G$ on $A$ satisfies the spectral subspace assumption.
When the action of $G$ on $A$ satisfies the spectral subspace condition, 
we call 
the Kasparov module $(\oplus_\chi A_\chi,(E_1)_{A^G\wh{\otimes}\bC},\D_1)$ the \emph{left-hand module}.
\end{prop}
\begin{proof}
See \cite{CNNR}*{Proposition 2.9} and the preceding lemmas for a proof
when $G=\T$. 
The general case requires only minor modifications, as in \cite{CGRS2}*{Chapter 5}.
\end{proof}
We henceforth assume that the action of $G$ on $A$ satisfies the spectral subspace assumption.
\subsection{The middle module.}
Recall that $G$ is a compact abelian Lie group, 
equipped with the 
trivial spinor bundle $\bs{\$}_G$, and $(\A,\h,\D)$ is 
an even $G$-equivariant spectral triple for a $\Z_2$-graded separable $C^*$-algebra $A$.
We will now construct the middle module, whose job is to correct for 
the spinor bundle dimensions between the left hand module and $(\A,\h,\D)$. 

Let $\Gamma(\bs{\$}_G)$ denote the continuous sections of 
$\bs{\$}_G$, which is a right Hilbert $C(G)$-module 
with the pointwise inner product on $\bs{\$}_G$. We also let 
$\Gamma(\Cl(G))$ denote the $C^*$-algebra of continuous sections 
of the Clifford bundle over $G$. 
Let $\rho:\Gamma(\Cl(G))\rightarrow\End_{C(G)}(\Gamma(\bs{\$}_G))$ be the 
Clifford representation, which is a $\ast$-homomorphism. 
When $G$ is even dimensional, $\rho$ is a 
$\Z_2$-graded $\ast$-homomorphism, 
but this is not the case when $G$ is odd dimensional. 
Both $\Gamma(\Cl(G))$ and $\Gamma(\bs{\$}_G)$ carry an action of $G$, 
and we denote their respective fixed point sets by $\Gamma(\Cl(G))^G$ 
and $\Gamma(\bs{\$}_G)^G$. The fixed point algebra 
$\Gamma(\Cl(G))^G$ is a finite dimensional $C^*$-algebra, with
$$
\Gamma(\Cl(G))^G\cong\Cl_n\cong
\left\{\begin{array}{ll}M_{2^{\dim G/2}}(\C)&\text{ if }G\text{ is even dimensional}\\M_{2^{(\dim G-1)/2}}(\C)\oplus M_{2^{(\dim G-1)/2}}(\C)&\text{ if }G\text{ is odd dimensional.}\end{array}\right.
$$
The fixed sections $\Gamma(\bs{\$}_G)^G$ form a finite dimensional vector space, with 
$$
\Gamma(\bs{\$}_G)^G\cong\left\{\begin{array}{ll}\C^{2^{\dim G/2}}&\text{ if }G\text{ is even dimensional}\\
\C^{2^{(\dim G-1)/2}}&\text{ if }G\text{ is odd dimensional.}\end{array}\right.
$$
Let $\bc$ be the generator of the $C^*$-algebra $\bC$, as in Definition \ref{defn:bc}. The 
$\Z_2$-graded $\ast$-homo\-morphism 
$\wt{\rho}:\Gamma(\Cl(G))^G\rightarrow\End_{\bC}(\Gamma(\bs{\$}_G)^G\wh{\otimes}\bC)$ 
defined on elements of homogeneous degree by 
\begin{equation}
\wt{\rho}(s)(w\wh{\otimes}z)=\rho(s)w\wh{\otimes}\bc^{\deg s}z,
\label{eq:cliff-iso}
\end{equation}
is an isomorphism.

The isomorphism \eqref{eq:cliff-iso} implies that 
$\Gamma(\bs{\$}_G)^G\wh{\otimes}\bC$ is a 
$\Z_2$-graded Morita equivalence bimodule between 
$\Gamma(\Cl(G))^G$ and $\bC$, where the left inner product is defined by
$$
\wt{\rho}\big(\,\!_{\Gamma(\Cl(G))^G}(w_1|w_2)\big)w_3=w_1(w_2|w_3)_{\bC}.
$$
Hence the conjugate module $(\Gamma(\bs{\$}_G)^G\wh{\otimes}\bC)^*$, \cite{RW}*{p. 49} is a 
$\Z_2$-graded Morita equivalence bimodule between $\bC$ and $\Gamma(\Cl(G))^G$.

The fixed point algebra $A^G$ is a $\Z_2$-graded right Hilbert 
module over itself, and left multiplication on itself defines a 
$\Z_2$-graded $\ast$-homomorphism $A^G\rightarrow \End_{A^G}(A^G)$. 

The external tensor product $A^G\wh{\otimes}(\Gamma(\bs{\$}_G)^G\wh{\otimes}\bC)^*$ 
is a $\Z_2$-graded right Hilbert $A^G\wh{\otimes}\Gamma(\Cl(G))^G$-module, 
which carries a representation 
$A^G\wh{\otimes}\bC\rightarrow\End_{A^G\wh{\otimes}\Gamma(\Cl(G))^G}(A^G\wh{\otimes}(\Gamma(\bs{\$}_G)^G\wh{\otimes}\bC)^*)$. 
Since $A^G\wh{\otimes}(\Gamma(\bs{\$}_G)^G\wh{\otimes}\bC)^*$ 
is a Morita equivalence bimodule, the triple 
$$
(A^G\wh{\otimes}\bC,(A^G\wh{\otimes}(\Gamma(\bs{\$}_G)^G\wh{\otimes}\bC)^*)_{A^G\wh{\otimes}\Gamma(\Cl(G))^G},0)
$$
is an (unbounded) equivariant Kasparov 
$A^G\wh{\otimes}\bC$-$A^G\wh{\otimes}\Gamma(\Cl(G))^G$-module, 
where the $C^*$-algebras and the Hilbert module carry the trivial action by $G$. We
call this module the \emph{middle module}.
\subsection{The right-hand module.}
To define the right-hand module we require greater compatibility between the action 
$\alpha$ of $G$ on $A$
and $\A\subset A$ than we have assumed so far. We say that $\A$ is \emph{$\alpha$-compatible} if
$$
\A_\chi:=\A\cap A_\chi\mbox{ is dense in }A_\chi\mbox{ for all }\chi\in\Char(G).
$$
Compatibility is implied by $\alpha$ restricting to a continuous action on $\A$ for some 
finer complete topology on $\A$.
\begin{defn}\label{defn:hchi}
For each $\chi\in\Char(G)$, let $\h_\chi=\{\xi\in\h:V_g\xi=\chi(g)\xi\}$ be the 
spectral subspace corresponding to $\chi$,  
and define an operator $\D_\chi:\dom(\D)\cap\h_\chi\subset\h_\chi\rightarrow\h_\chi$ by 
$\D_\chi\xi:=\D\xi$. The Hilbert space $\h_\chi$ inherits the $\Z_2$-grading of $\h$.
\end{defn}
\begin{lemma}\label{lem:fixedtriple}
Suppose that $\A$ is $\alpha$-compatible.
Let $\A^G$ be the fixed point algebra of $\A$. Then for each 
$\chi\in\Char(G)$, $(\A^G,\h_\chi,\D_\chi)$ is an even 
equivariant spectral triple for $A^G$, where $\h_\chi$ inherits the action of $G$ on $\h$.
\end{lemma}
\begin{proof}
Since $G$ acts on $\h$ unitarily, there is an orthogonal 
decomposition $\h=\bigoplus_{\chi\in\Char(G)}\h_\chi$. 
The density of $\dom(\D)$ in $\h$ thus implies that 
$\dom(\D_\chi)$ is dense in $\h_\chi$ for all $\chi\in\Char(G)$.

The operator $(1+\D^2)^{-1/2}\in B(\h)$ is self-adjoint, and since 
$\D$ commutes with the action of $G$, so too does $(1+\D^2)^{-1/2}$. 
Hence $(1+\D^2)^{-1/2}|_{\h_\chi}$ is a bounded self-adjoint operator 
on $\h_\chi$, and $(1+\D^2)^{-1/2}|_{\h_\chi}=(1+\D_\chi^2)^{-1/2}$ 
for all $\chi\in\Char(G)$. Hence
$$
F_\chi:=\D(1+\D^2)^{-1/2}|_{\h_\chi}=\D_\chi(1+\D_\chi^2)^{-1/2}
$$
is also a bounded self-adjoint operator on $\h_\chi$. Since 
$\D_\chi=F_\chi(1-F_\chi^2)^{-1/2}$, it follows from 
\cite{Lance}*{Theorem 10.4} that $\D_\chi$ is a self-adjoint operator on $\h_\chi$.

Since $[\D_\chi,a]=[\D,a]|_{\h_\chi}$ and 
$a(1+\D_\chi^2)^{-1/2}=a(1+\D^2)^{-1/2}|_{\h_\chi}$ for all 
$a\in\A^G$, it follows that $(\A^G,\h_\chi,\D_\chi)$ satisfies 
the conditions of Definition \ref{defn:kasmodule}, and hence 
$(\A^G,\h_\chi,\D_\chi)$ is an even equivariant spectral triple.
\end{proof}
We wish to use the operator $\D_\zeta$ to construct our final Kasparov module, 
for some fixed $\zeta\in\Char(G)$. However, the middle module is an 
unbounded Kasparov $A^G$-$A^G\wh{\otimes}\Gamma(\Cl(G))^G$-module, 
whereas $(\A^G,\h_\zeta,\D_\zeta)$ is an unbounded Kasparov $A^G$-$\C$-module. 
Hence we need a representation of $\Gamma(\Cl(G))^G$ on $\h_\zeta$, 
which will define an action of $A^G\wh{\otimes}\Gamma(\Cl(G))^G$ on $\h_\zeta$. 
The conditions we impose below on the action and the character $\zeta$
ensure that we obtain an even 
spectral triple for $A^G\wh{\otimes}\Gamma(\Cl(G))^G$, and 
in addition that
Kucerovsky's connection criteria is satisfied (Proposition \ref{prop:connection}).

Simple examples show that $\h_\chi$ may be trivial for any given $\chi\in\Char(G)$, 
including the trivial character $\chi(g)=1$. 
We therefore impose the condition $\ol{A\h_\zeta}=\h$ 
on the character $\zeta$ in order to construct the right-hand module. 
Choosing $\zeta$ in this way allows us to recover the original Hilbert space $\h$ from the three modules. 
\begin{rem}
Even if $\ol{A\h_\chi}=\h$ for all $\chi\in\Char(G)$, 
the positivity criterion may be satisfied for some choices of $\zeta$ but not for others. For an
example see
Section \ref{sec:dirac-ess2}.
\end{rem}
\begin{defn}
\label{defn:theta}
Suppose that $\A$ is $\alpha$-compatible.
Let $\zeta\in\Char(G)$ be such that $\ol{A\h_\zeta}=\h$, and let $\eta:\Gamma(\Cl(G))^G\rightarrow B(\h)$ be a unital, equivariant $\Z_2$-graded $\ast$-homomorphism such that\\
1) $[\eta(s),a]_\pm=0$ for all $s\in\Gamma(\Cl(G))^G$ and $a\in A^G$, and\\
2) $a\eta(s)\cdot\dom(\D_\zeta)\subset\dom(\D)$ and $[\D,\eta(s)]_\pm aP_\zeta$ 
is bounded on $\h$ for all $a\in\oplus_\chi\A_\chi$ and \\$s\in\Gamma(\Cl(G))^G$, where 
$P_\zeta\in B(\h)$ is the orthogonal projection onto $\h_\zeta$. 
\end{defn}

We define a $\Z_2$-graded $\ast$-homomorphism 
$A^G\wh{\otimes}\Gamma(\Cl(G))^G\rightarrow B(\h_\zeta)$ 
by $(a\wh{\otimes}s)\cdot\xi:=a\eta(s)\xi$. 
If $\A$ is $\alpha$-compatible, the conditions on $\eta$ and Lemma \ref{lem:fixedtriple} 
ensure that $(\A^G\wh{\otimes}\Gamma(\Cl(G))^G,\h_\zeta,\D_\zeta)$ 
is an even equivariant spectral triple for $A^G$, which we call the \emph{right-hand module}.
\begin{rem}
Condition 2) of Definition \ref{defn:theta} is stronger than necessary to 
ensure that we obtain an equivariant spectral triple for 
$A^G\wh{\otimes}\Gamma(\Cl(G))^G$, but this stronger 
condition is sufficient to prove that Kucerovsky's connection criteria is satisfied.
\end{rem}
\section{The Kasparov product of the left-hand, middle and right-hand modules.}
Recall that $G$ is a compact abelian Lie group, equipped with the normalised Haar measure and a trivial spinor bundle $\bs{\$}_G$, and $(\A,\h,\D)$ is an even $G$-equivariant spectral triple for a $\Z_2$-graded separable $C^*$-algebra $A$. Let $\zeta\in\Char(G)$ and $\eta:\Gamma(\Cl(G))^G\rightarrow B(\h)$ satisfy the conditions of Definition \ref{defn:theta}, so in particular $\A$ is $\alpha$-compatible.

The next result can be proved with a straightforward application of 
Kucerovsky's criteria,  \cite{Kucerovsky}*{Theorem 13}. 
\begin{prop}
\label{prop:leftmiddle}
The Kasparov product of the left-hand and middle modules is 
represented by 
$(\oplus_\chi A_\chi,(E_1\wh{\otimes}_{A^G\wh{\otimes}\bC}(A^G\wh{\otimes}(\Gamma(\bs{\$}_G)^G\wh{\otimes}\bC)^*))_{A^G\wh{\otimes}\Gamma(\Cl(G))^G},\D_1\wh{\otimes}1)$.
\end{prop}
To determine whether the Kasparov product of the left-hand, middle and right-hand modules is represented by 
$(\A,\h,\D)$, we first construct an isomorphism 
$$
\Psi:(E_1\wh{\otimes}_{A^G\wh{\otimes}\bC}(A^G\wh{\otimes}(\Gamma(\bs{\$}_G)^G\wh{\otimes}\bC)^*))\wh{\otimes}_{A^G\wh{\otimes}\Gamma(\Cl(G))^G}\h_\zeta\to \h,
$$ 
which will allow us to use Kucerovsky's criteria,  \cite{Kucerovsky}*{Theorem 13}. 
We would like to define the map $\Psi$ on elements of homogeneous degree by
\begin{align}
\Psi\Big(\big((y\wh{\otimes}u)\wh{\otimes}(a\wh{\otimes}\ol{w})\big)\wh{\otimes}\xi\Big)
:=(-1)^{\deg u\cdot\deg a}\sum_{\chi\in\Char(G)}\Phi_\chi(y)a\eta\big(\,\!_{\Gamma(\Cl(G))^G}(\chi^{-1} p_{\chi^{-1}} u|w)\big)\xi,
\label{eq:sigh}
\end{align}
where $p_\chi\in\End_{\bC}(L^2(\bs{\$}_G)\wh{\otimes}\bC)$ and $\Phi_\chi\in\End_{A^G}(X)$ are the spectral subspace projections of Definition \ref{defn:pchi} and Lemma \ref{lem:innerproductsum} respectively. 

To see that $\Psi$ is well-defined, even on homogeneous elements, we need to know that the sum
over characters converges. This is established by the following lemma.
\begin{lemma}
\label{lem:isometric}
Let $\big((y_i\wh{\otimes}u_i)\wh{\otimes}(a_i\wh{\otimes}\ol{w_i})\big)\wh{\otimes}\xi_i\in(E_1\wh{\otimes}_{A^G\wh{\otimes}\bC}(A^G\wh{\otimes}(\Gamma(\bs{\$}_G)^G\wh{\otimes}\bC)^*))\wh{\otimes}_{A^G\wh{\otimes}\Gamma(\Cl(G))^G}\h_\zeta$ for $i=1,2$. Then
\begin{align*}
&\Ideal{\Psi\Big(\big((y_1\wh{\otimes}u_1)\wh{\otimes}(a_1\wh{\otimes}\ol{w_1})\big)\wh{\otimes}\xi_1\Big),\Psi\Big(\big((y_2\wh{\otimes}u_2)\wh{\otimes}(a_2\wh{\otimes}\ol{w_2})\big)\wh{\otimes}\xi_2\Big)}\\
&\qquad\qquad=\Ideal{\big((y_1\wh{\otimes}u_1)\wh{\otimes}(a_1\wh{\otimes}\ol{w_1})\big)\wh{\otimes}\xi_1,\big((y_2\wh{\otimes}u_2)\wh{\otimes}(a_2\wh{\otimes}\ol{w_2})\big)\wh{\otimes}\xi_2}
\end{align*}
and hence $\Psi$ is a well-defined isometry.
\end{lemma}
\begin{proof}
Suppose that both elements are of homogeneous degree. Then using Lemma \ref{lem:prodesp},
\begin{align*}
&\Ideal{\big((y_1\wh{\otimes}u_1)\wh{\otimes}(a_1\wh{\otimes}\ol{w_1})\big)\wh{\otimes}\xi_1,\big((y_2\wh{\otimes}u_2)\wh{\otimes}(a_2\wh{\otimes}\ol{w_2})\big)\wh{\otimes}\xi_2}\\
&=(-1)^{\deg u_1\cdot(\deg y_1+\deg y_2)+(\deg u_1+\deg u_2)\cdot\deg a_2+\deg w_1\cdot(\deg a_1+\deg y_1+\deg y_2+\deg a_2)}\\
&\qquad\times\sum_{\chi\in\Char(G)}\Ideal{\xi_1,a_1^*\Phi_\chi(y_1)^*\Phi_\chi(y_2)a_2\eta\big(\,\!_{\Gamma(\Cl(G))^G}(w_1|w_2(p_{\chi^{-1}} u_2|p_{\chi^{-1}} u_1)_{\bC})\big)\xi_2}\\
&=(-1)^{\deg u_1\cdot\deg a_1+\deg u_2\cdot\deg a_2}\\
&\times\sum_{\chi\in\Char(G)}\Ideal{\Phi_\chi(y_1)a_1\eta\big(\,\!_{\Gamma(\Cl(G))^G}(\chi^{-1} p_{\chi^{-1}} u_1|w_1)\big)\xi_1,\Phi_\chi(y_2)a_2\eta\big(\,\!_{\Gamma(\Cl(G))^G}(\chi^{-1} p_{\chi^{-1}} u_2|w_2)\big)\xi_2}\\
&=\Ideal{\Psi\Big(\big((y_1\wh{\otimes}u_1)\wh{\otimes}(a_1\wh{\otimes}\ol{w_1})\big)\wh{\otimes}\xi_1\Big),\Psi\Big(\big((y_2\wh{\otimes}u_2)\wh{\otimes}(a_2\wh{\otimes}\ol{w_2})\big)\wh{\otimes}\xi_2\Big)}.
\end{align*}
The penultimate line follows from
\begin{align}\label{eq:carthage}
\,\!_{\Gamma(\Cl(G))^G}(w_1|\chi^{-1} p_{\chi^{-1}} u_1)\,\!_{\Gamma(\Cl(G))^G}(\chi^{-1} p_{\chi^{-1}} u_2|w_2)=\,\!_{\Gamma(\Cl(G))^G}(w_1|w_2(p_{\chi^{-1}} u_2|p_{\chi^{-1}} u_1)_{\bC}),
\end{align}
which in turn follows from $(\chi^{-1} p_{\chi^{-1}} u_2|\chi^{-1} p_{\chi^{-1}} u_1)_{\bC}=(p_{\chi^{-1}} u_2|p_{\chi^{-1}} u_1)_{\bC}$.

We have already established that the sum 
$\sum_{\chi\in\Char(G)}\Phi_\chi(y)a\eta\big(\,\!_{\Gamma(\Cl(G))^G}(\chi^{-1} p_{\chi^{-1}} u|w)\big)\xi$ 
converges. It only remains to check that $\Psi$ is 
well-defined with respect to the balanced 
tensor products, which is a straightforward exercise.
\end{proof}
\begin{prop}
The map $\Psi$ is a unitary, equivariant, $\Z_2$-graded, $A$-linear isomorphism. The inverse
$$
\Psi^{-1}:\h\rightarrow(E_1\wh{\otimes}_{A^G\wh{\otimes}\bC}(A^G\wh{\otimes}(\Gamma(\bs{\$}_G)^G\wh{\otimes}\bC)^*))\wh{\otimes}_{A^G\wh{\otimes}\Gamma(\Cl(G))^G}\h_\zeta
$$
is defined as follows. Let $(x_j)_{j=1}^n$ be a $G$-invariant global 
orthonormal frame for $\bs{\$}_G$,  and let $(\phi_\ell)_{\ell=1}^\infty$ be an approximate identity for $A^G$ of homogeneous degree zero. For $\xi\in\h$, choose sequences $(a_k)_{k=1}^\infty\subset A$ and $(\xi_k)_{k=1}^\infty\subset\h_\zeta$ such that $a_k\xi_k\rightarrow\xi$ as $k\rightarrow\infty$. Then
$$
\Psi^{-1}(\xi):=\sum_{\chi\in\Char(G)}\sum_{j=1}^{n}\lim_{k\rightarrow\infty}\lim_{\ell\rightarrow\infty}\big(\big(\Phi_\chi(a_k)\wh{\otimes}(\chi x_j\wh{\otimes}1)\big)\wh{\otimes}(\phi_\ell\wh{\otimes}\ol{x_j\wh{\otimes}1})\big)\wh{\otimes}\xi_k.
$$
\end{prop}
\begin{proof}
It is immediate that $\Psi$ is equivariant 
and $\Z_2$-graded, and $\Psi$ is an isometry by Lemma \ref{lem:isometric}. 
So it remains  to show that (i) 
$\Psi$ is $A$-linear, and (ii)  $\Psi^{-1}$ is an inverse for $\Psi$.

(i) Let $b\in A$. Then
\begin{align*}
&\Psi\Big(b\cdot\big((y\wh{\otimes}u)\wh{\otimes}(a\wh{\otimes}\ol{w})\big)\wh{\otimes}\xi\Big)=\sum_{\mu\in\Char(G)}\Psi\Big(\big((\Phi_\mu(b)y\wh{\otimes}\mu u)\wh{\otimes}(a\wh{\otimes}\ol{w})\big)\wh{\otimes}\xi\Big)\\
&=(-1)^{\deg u\cdot\deg a}\sum_{\chi,\mu\in\Char(G)}\Phi_\chi(\Phi_\mu(b)y)a\eta\big(\,\!_{\Gamma(\Cl(G))^G}(\chi^{-1} p_{\chi^{-1}}\mu u|w)\big)\xi\\
&=(-1)^{\deg u\cdot\deg a}\sum_{\chi,\mu}\Phi_\mu(b)\Phi_{\chi}(y)a\eta\big(\,\!_{\Gamma(\Cl(G))^G}(\chi^{-1} p_{\chi^{-1}}u|w)\big)\xi=b\Psi\Big(\big((y\wh{\otimes}u)\wh{\otimes}(a\wh{\otimes}\ol{w})\big)\wh{\otimes}\xi\Big),
\end{align*}
so $\Psi$ is $A$-linear.

(ii) We first check that $\Psi^{-1}$ is well-defined, 
which means checking that the limits exist and that the sum converges. 
Suppose $\xi\in\h$, and choose sequences $(a_k)_{k=1}^\infty\subset A$ 
and $(\xi_k)_{k=1}^\infty\subset\h_\zeta$ such that $a_k\xi_k\rightarrow\xi$ 
as $k\rightarrow\infty$, which exist since $\ol{A\h_\zeta}=\h$. 
Since $\sum_{j=1}^n\,\!_{\Gamma(\Cl(G))^G}(x_j\wh{\otimes}1|x_j\wh{\otimes}1)=1$,
\begin{align*}
\Psi\left(\sum_{j=1}^{n}\big(\big(\Phi_\chi(a_k)\wh{\otimes}(\chi x_j\wh{\otimes}1)\big)\wh{\otimes}(\phi_\ell\wh{\otimes}\ol{x_j\wh{\otimes}1})\big)\wh{\otimes}\xi_k\right)&=\sum_{j=1}^{n}\Phi_\chi (a_k)\phi_\ell\eta\big(\,\!_{\Gamma(\Cl(G))^G}(x_j\wh{\otimes}1|x_j\wh{\otimes}1)\big)\xi_k\\
&=\Phi_\chi(a_k)\phi_\ell\xi_k=P_{\chi\zeta}(a_k\phi_\ell\xi_k),
\end{align*}
where $P_{\chi\zeta}\in B(\h)$ is the orthogonal projection onto $\h_{\chi\zeta}$, and
$$\lim_{k\rightarrow\infty}\lim_{\ell\rightarrow\infty}P_{\chi\zeta}(a_k\phi_\ell\xi_k)=\lim_{k\rightarrow\infty}P_{\chi\zeta}(a_k\xi_k)=P_{\chi\zeta}\xi.$$
Since $\Psi$ is an isometry, this establishes that the limits exist.
Moreover,
$$\sum_{\chi\in\Char(G)}P_{\chi\zeta}\xi=\sum_{\chi\in\Char(G)}P_\chi\xi=\xi,$$
so the sum converges. This calculation also shows that $\Psi^{-1}$ is a right-inverse for $\Psi$, so  
that $\Psi$ is surjective. Since $\Psi$ is injective, it follows that $\Psi$ is invertible with inverse $\Psi^{-1}$.
\end{proof}
Now that we have the isomorphism $\Psi$, we can use Kucerovsky's criteria, 
\cite{Kucerovsky}*{Theorem 13}, to determine if $(\A,\h,\D)$ represents 
the Kasparov product of the left-hand, middle and right-hand modules. 
More precisely, $(\A,\h,\D)$ is unitarily equivalent as an unbounded equivariant Kasparov module to 
$(\A,(E_1\wh{\otimes}_{A^G\wh{\otimes}\bC}(A^G\wh{\otimes}(\Gamma(\bs{\$}_G)^G\wh{\otimes}\bC)^*))\wh{\otimes}_{A^G\wh{\otimes}\Gamma(\Cl(G))^G}\h_\zeta,\Psi^{-1}\circ\D\circ\Psi)$, 
and Kucerovsky's criteria may now be applied to determine whether factorisation has been achieved.

\begin{thm}[The criterion for factorisation]
\label{thm:main}
Let $\zeta\in\Char(G)$ and $\eta:\Gamma(\Cl(G))^G\rightarrow B(\h)$ satisfy the conditions of Definition \ref{defn:theta}, so in particular $\A$ is $\alpha$-compatible. 
Let $(x_j)_{j=1}^{n}$ be a $G$-invariant global orthonormal frame for $\bs{\$}_G$, and for each $\chi\in\Char(G)$, let $P_\chi\in B(\h)$ be the orthogonal projection onto $\h_\chi$. 
If there is some $R\in\mathbb{R}$ such that
\begin{align}
&\sum_{j=1}^{n}\Big(\Ideal{\D\xi,\eta\big(\,\!_{\Gamma(\Cl(G))^G}(\chi^{-1}\D_G(\chi x_j)\wh{\otimes}\bc|x_j\wh{\otimes}1)\big)P_{\chi\zeta}\xi}\nonumber\\
&\qquad+\Ideal{\eta\big(\,\!_{\Gamma(\Cl(G))^G}(\chi^{-1}\D_G(\chi x_j)\wh{\otimes}\bc|x_j\wh{\otimes}1)\big)P_{\chi\zeta}\xi,\D\xi}\Big)\geq R\|\xi\|^2
\label{eq:pos}
\end{align}
for all $\chi\in\Char(G)$, $\xi\in\dom(\D)$, then $(\A,\h,\D)$ represents the Kasparov product of left-hand, middle and right-hand modules.
\end{thm}
\begin{rem}
Although  \cite{Kucerovsky}*{Theorem 13} is stated for the non-equivariant case, 
it requires no modification in the equivariant case, \cite{Kucerovsky2000}.
\end{rem}
Theorem \ref{thm:main} is proved by showing that Kucerovsky's domain and connection
conditions hold under the existing assumptions. The remaining positivity 
condition is precisely condition \eqref{eq:pos}. 
\begin{prop}[The connection criterion]
\label{prop:connection}
For each $e\in E_1\wh{\otimes}_{A^G\wh{\otimes}\bC}(A^G\wh{\otimes}(\Gamma(\bs{\$}_G)^G\wh{\otimes}\bC)^*)$, let $T_e:\h_\zeta\rightarrow(E_1\wh{\otimes}_{A^G\wh{\otimes}\bC}(A^G\wh{\otimes}(\Gamma(\bs{\$}_G)^G\wh{\otimes}\bC)^*))\wh{\otimes}_{A^G\wh{\otimes}\Gamma(\Cl(G))^G}\h_\zeta$
be the creation operator.
The graded commutators
$$\left[\begin{pmatrix}\Psi^{-1}\circ\D\circ\Psi&0\\0&\D_\zeta\end{pmatrix},\begin{pmatrix}0&T_e\\T_e^*&0\end{pmatrix}\right]_\pm$$
are bounded for all $e\in Y$, where 
$Y\subset E_1\wh{\otimes}_{A^G\wh{\otimes}\bC}(A^G\wh{\otimes}(\Gamma(\bs{\$}_G)^G\wh{\otimes}\bC)^*)$ 
is the dense subspace
\begin{align*}
Y&:=\Span\{(z\wh{\otimes}s)\wh{\otimes}(a\wh{\otimes}\ol{w})\in E_1\wh{\otimes}_{A^G\wh{\otimes}\bC}(A^G\wh{\otimes}(\Gamma(\bs{\$}_G)^G\wh{\otimes}\bC)^*):z\in \oplus_\chi \A_\chi,\ a\in\A^G\}.
\end{align*}
\end{prop}
\begin{proof}
Let $e=(z\wh{\otimes}s)\wh{\otimes}(a\wh{\otimes}\ol{w})\in Y$, 
$\big((y\wh{\otimes}t)\wh{\otimes}(b\wh{\otimes}\ol{v})\big)\wh{\otimes}\xi\in\dom(\Psi^{-1}\circ\D\circ\Psi)$ 
and 
$\psi\in\dom(\D_\zeta)$, each of homogeneous degree. Then the upper entry of the column vector 
$$
\left[\begin{pmatrix}\Psi^{-1}\circ\D\circ\Psi&0\\0&\D_\zeta\end{pmatrix},\begin{pmatrix}0&T_e\\T_e^*&0\end{pmatrix}\right]_\pm\begin{pmatrix}\big((y\wh{\otimes}t)\wh{\otimes}(b\wh{\otimes}\ol{v})\big)\wh{\otimes}\xi\\\psi\end{pmatrix}
$$ 
is
\begin{align*}
&\Psi^{-1}\circ\D\circ\Psi\circ T_e\psi-(-1)^{\deg z+\deg s+\deg a+\deg w}T_e\circ\D_\zeta\psi\\
&=\Psi^{-1}\circ\D\circ\Psi\Big(\big((z\wh{\otimes}s)\wh{\otimes}(a\wh{\otimes}\ol{w})\big)\wh{\otimes}\psi\Big)-(-1)^{\deg z+\deg s+\deg a+\deg w}\big((z\wh{\otimes}s)\wh{\otimes}(a\wh{\otimes}\ol{w})\big)\wh{\otimes}\D_\zeta\psi\\
&=(-1)^{\deg s\cdot\deg a}\Psi^{-1}\circ\D\sum_{\chi\in\Char(G)}\Phi_\chi(z) a\eta\big(\,\!_{\Gamma(\Cl(G))^G}(\chi^{-1} p_{\chi^{-1}} s|w)\big)\psi\\
&\quad-(-1)^{\deg z+\deg s+\deg a+\deg w+\deg s\cdot\deg a}\Psi^{-1}\sum_{\chi\in\Char(G)}\Phi_\chi(y)a\eta\big(\,\!_{\Gamma(\Cl(G))^G}(\chi^{-1} p_{\chi^{-1}} s|w)\big)\D_\zeta\psi\\
&=(-1)^{\deg s\cdot\deg a}\Psi^{-1}\sum_{\chi\in\Char(G)}[\D,\Phi_\chi(z)a\eta\big(\,\!_{\Gamma(\Cl(G))^G}(\chi^{-1} p_{\chi^{-1}} s|w)\big)]_\pm \psi,
\end{align*}
and we estimate
\begin{align*}
&\bigg\|(-1)^{\deg s\cdot\deg a}\Psi^{-1}\sum_{\chi\in\Char(G)}[\D,\Phi_\chi(z)a\eta\big(\,\!_{\Gamma(\Cl(G))^G}(\chi^{-1} p_{\chi^{-1}} s|w)\big)]_\pm \psi\bigg\|^2\\
&=\sum_{\chi\in\Char(G)}\big\|[\D,\Phi_\chi(z)a\eta\big(\,\!_{\Gamma(\Cl(G))^G}(\chi^{-1} p_{\chi^{-1}} s|w)\big)]_\pm P_\zeta\psi\big\|^2\\
&\leq\|\psi\|^2\sum_{\chi\in\Char(G)}\big\|[\D,\Phi_\chi(z)a\eta\big(\,\!_{\Gamma(\Cl(G))^G}(\chi^{-1} p_{\chi^{-1}} s|w)\big)]_\pm P_\zeta\big\|^2,
\end{align*}
where the sum converges since $z\in \oplus_\chi\A_\chi$. 
Hence the upper entry is a bounded function of $\psi$.
For the lower entry we have
\begin{align*}
&\D_\zeta\circ T_e^*\big(\big((y\wh{\otimes}t)\wh{\otimes}(b\wh{\otimes}v)\big)\wh{\otimes}\xi\big)=\D_\zeta\big(\big((z\wh{\otimes} s)\wh{\otimes}(a\wh{\otimes}\ol{w})\big|(y\wh{\otimes} t)\wh{\otimes}(b\wh{\otimes}v)\big)_{A^G\wh{\otimes}\Gamma(\Cl(G))^G}\xi\big)\\
&=(-1)^{\deg s\cdot(\deg z+\deg y)+\deg w\cdot(\deg a+\deg z+\deg y+\deg b)+\deg b\cdot(\deg s+\deg t)}\\
&\quad\qquad\qquad\times\sum_{\chi\in\Char(G)}\D_\zeta\Big(a^*\Phi_\chi(z)^*\Phi_\chi(y)b\eta\big(\,\!_{\Gamma(\Cl(G))^G}(\chi^{-1} p_{\chi^{-1}}s|w)\,\!_{\Gamma(\Cl(G))^G}(\chi^{-1} p_{\chi^{-1}}t|v)\big)\xi\Big),
\end{align*}
using Lemma \ref{lem:prodesp} and Equation \eqref{eq:carthage}. 
Let $(x_j)_{j=1}^{n}$ be a $G$-invariant, global orthonormal frame 
for $\bs{\$}_G$, and let $(\phi_\ell)_{\ell=1}^\infty$ be an approximate 
identity for $A^G$ of homogeneous degree zero. For each 
$\chi\in\Char(G)$, let $(c_k^\chi)_{k=1}^\infty\subset A$ and 
$(\sigma_k^\chi)_{k=1}^\infty\subset\h_\zeta$ be sequences such that
$$
\lim_{k\rightarrow\infty}c_k^\chi \sigma_k^\chi
=\D\big(\Phi_\chi(y) b\eta\big(\,\!_{\Gamma(\Cl(G))^G}(\chi^{-1} p_{\chi^{-1}} t|v)\big)\xi\big).
$$
Then
\begin{align*}
&T_e^*\circ\Psi^{-1}\circ\D\circ\Psi\big(\big((y\wh{\otimes}t)\wh{\otimes}(b\wh{\otimes}v)\big)\wh{\otimes}\xi\big)=(-1)^{\deg t\cdot\deg b+\deg s\cdot\deg z+\deg w\cdot(\deg a+\deg z)}\\
&\quad\qquad\qquad\times\sum_{\chi,\nu\in\Char(G)}\sum_{j=1}^{n}\lim_{k\rightarrow\infty}a^*\Phi_\nu(z)^*\eta\big(\,\!_{\Gamma(\Cl(G))^G}(w|(x_j\wh{\otimes}1)\cdot(\nu x_j\wh{\otimes}1|p_{\nu^{-1}} s)_{\bC})\big)\Phi_\nu(c_k^\chi)\sigma_k^\chi\\
&=(-1)^{\deg t\cdot\deg b+\deg s\cdot\deg z+\deg w\cdot(\deg a+\deg z)}\times\\
&\sum_{\chi}\sum_{j=1}^{n}a^*\Phi_\chi(z)^*\eta\big(\,\!_{\Gamma(\Cl(G))^G}(w|(x_j\wh{\otimes}1)\cdot(\chi x_j\wh{\otimes}1|p_{\chi^{-1}} s)_{\bC})\big)\D\Big(\Phi_\chi(y) b\eta\big(\,\!_{\Gamma(\Cl(G))^G}(\chi^{-1} p_{\chi^{-1}} t|v)\big)\xi\Big)
\end{align*}
where we have used
\begin{align*}
\lim_{k\rightarrow\infty}\Phi_\nu(c_k^\chi)\sigma^\chi_k&=\lim_{k\rightarrow\infty}P_{\nu\zeta}c^\chi_k\sigma^\chi_k=P_{\nu\zeta}\D\big(\Phi_\chi(y) b\eta\big(\,\!_{\Gamma(\Cl(G))^G}(\chi^{-1} p_{\chi^{-1}} t|v)\big)\xi\big)\\
&=\delta_{\nu,\chi}\D\big(\Phi_\chi(y) b\eta\big(\,\!_{\Gamma(\Cl(G))^G}(\chi^{-1} p_{\chi^{-1}} t|v)\big)\xi\big).
\end{align*}
Since $\chi^{-1} p_{\chi^{-1}} s=\sum_{j=1}^{n}(x_j\wh{\otimes}1)\cdot(\chi x_j\wh{\otimes}1|p_{\chi^{-1}} s)_{\bC}$,
\begin{align*}
&T_e^*\circ\Psi^{-1}\circ\D\circ\Psi\big(\big((y\wh{\otimes}t)\wh{\otimes}(b\wh{\otimes}v)\big)\wh{\otimes}\xi\big)=(-1)^{\deg t\cdot\deg b+\deg s\cdot\deg z+\deg w\cdot(\deg a+\deg z)}\\
&\quad\qquad\times\sum_{\chi\in\Char(G)}a^*\Phi_\chi(z)^*\eta\big(\,\!_{\Gamma(\Cl(G))^G}(w|\chi^{-1} p_{\chi^{-1}} s)\big)\D\Big(\Phi_\chi(y) b\eta\big(\,\!_{\Gamma(\Cl(G))^G}(\chi^{-1} p_{\chi^{-1}} t|v)\big)\xi\Big).
\end{align*}
Hence the lower entry is
\begin{align*}
&\D_\zeta\circ T_e^*\big(\big((y\wh{\otimes}t)\wh{\otimes}(b\wh{\otimes}v)\big)\wh{\otimes}\xi\big)-(-1)^{\deg z+\deg s+\deg a+\deg w}T_e^*\circ\Psi^{-1}\circ\D\circ\Psi\big(\big((y\wh{\otimes}t)\wh{\otimes}(b\wh{\otimes}v)\big)\wh{\otimes}\xi\big)\\
&=(-1)^{\deg s\cdot(\deg z+\deg y)+\deg w\cdot(\deg a+\deg z+\deg y+\deg b)+\deg b\cdot(\deg s+\deg t)}\\
&\quad\qquad\qquad\times\sum_{\chi\in\Char(G)}\D_\zeta\Big(a^*\Phi_\chi(z)^*\Phi_\chi(y)b\eta\big(\,\!_{\Gamma(\Cl(G))^G}(\chi^{-1} p_{\chi^{-1}}s|w)\,\!_{\Gamma(\Cl(G))^G}(\chi^{-1} p_{\chi^{-1}}t|v)\big)\xi\Big)\\
&\quad-(-1)^{\deg z+\deg s+\deg a+\deg w+\deg t\cdot\deg b+\deg s\cdot\deg z+\deg w\cdot(\deg a+\deg z)}\\
&\qquad\qquad\times\sum_{\chi\in\Char(G)}a^*\Phi_\chi(z)^*\eta\big(\,\!_{\Gamma(\Cl(G))^G}(w|\chi^{-1} p_{\chi^{-1}} s)\big)\D\Big(\Phi_\chi(y) b\eta\big(\,\!_{\Gamma(\Cl(G))^G}(\chi^{-1} p_{\chi^{-1}} t|v)\big)\xi\Big)\\
&=(-1)^{\deg t\cdot\deg b+\deg s\cdot\deg z+\deg w\cdot(\deg a+\deg z)}\\
&\qquad\qquad\quad\times\sum_{\chi\in\Char(G)}[\D,a^*\Phi_\chi(z)^*\eta\big(\,\!_{\Gamma(\Cl(G))^G}(w|\chi p_{\chi} s)\big)]_\pm\Phi_\chi(y) b\eta\big(\,\!_{\Gamma(\Cl(G))^G}(\chi p_{\chi} t|v)\big)\xi.
\end{align*}
Since $\Psi$ is an isometry, the sum $\sum_{\chi\in\Char(G)}\Phi_\chi(y)a\eta\big(\,\!_{\Gamma(\Cl(G))^G}(\chi^{-1} p_{\chi^{-1}} t|w)\big)\xi$ converges, so 
\begin{align*}
&\sum_{\chi\in\Char(G)}[\D,a^*\Phi_\chi(z)^*\eta\big(\,\!_{\Gamma(\Cl(G))^G}(w|\chi^{-1} p_{\chi^{-1}} s)\big)]_\pm\Phi_\chi(y) b\eta\big(\,\!_{\Gamma(\Cl(G))^G}(\chi^{-1} p_{\chi^{-1}} t|v)\big)\xi\\
&=\left(\sum_{\nu}P_\zeta[\D,a^*\Phi_\nu(z)^*\eta\big(\,\!_{\Gamma(\Cl(G))^G}(w|\nu^{-1} p_{\nu^{-1}} s)\big)]_\pm\right)\left(\sum_{\chi}\Phi_\chi(y) b\eta\big(\,\!_{\Gamma(\Cl(G))^G}(\chi^{-1} p_{\chi^{-1}} t|v)\big)\xi\right).
\end{align*}
Thus we can estimate the lower entry by
\begin{align*}
&\bigg\|\sum_{\chi\in\Char(G)}[\D,a^*\Phi_\chi(z)^*\eta\big(\,\!_{\Gamma(\Cl(G))^G}(w|\chi^{-1} p_{\chi^{-1}} s)\big)]_\pm\Phi_\chi(y) b\eta\big(\,\!_{\Gamma(\Cl(G))^G}(\chi^{-1} p_{\chi^{-1}} t|v)\big)\xi\bigg\|^2\\
&\leq\bigg\|\sum_{\nu}P_\zeta[\D,a^*\Phi_\nu(z)^*\eta\big(\,\!_{\Gamma(\Cl(G))^G}(w|\nu^{-1} p_{\nu^{-1}} s)\big)]_\pm\bigg\|^2\sum_{\chi}\big\|\Phi_\chi(y) b\eta\big(\,\!_{\Gamma(\Cl(G))^G}(\chi^{-1} p_{\chi^{-1}} t|v)\big)\xi\big\|^2\\
&=\bigg\|\sum_{\nu\in\Char(G)}P_\zeta[\D,a^*\Phi_\nu(z)^*\eta\big(\,\!_{\Gamma(\Cl(G))^G}(w|\nu^{-1} p_{\nu^{-1}} s)\big)]_\pm\bigg\|^2\big\|\big((y\wh{\otimes}t)\wh{\otimes}(b\wh{\otimes}\ol{v})\big)\wh{\otimes}\xi\big\|^2,
\end{align*}
since $\Psi$ is an isometry. We note that 
$\sum_{\nu\in\Char(G)}P_\zeta[\D,a^*\Phi_\nu(z)^*\eta\big(\,\!_{\Gamma(\Cl(G))^G}(w|\nu^{-1} p_{\nu^{-1}} s)\big)]_\pm$ 
is a finite sum of bounded operators and hence is bounded. 
Therefore the lower entry is a bounded function of 
$\big((y\wh{\otimes}t)\wh{\otimes}(b\wh{\otimes}\ol{v})\big)\wh{\otimes}\xi$.
\end{proof}
\begin{lemma}\label{lem:d111}
Let $(x_j)_{j=1}^{n}$ be a $G$-invariant global orthonormal frame 
for $\bs{\$}_G$, let $\D_G$ be the Dirac operator on $\bs{\$}_G$, 
and let $P_\chi\in B(\h)$ be the projection onto $\h_\chi$ for 
$\chi\in\Char(G)$. Then
$$
\Psi\circ(\D_1\wh{\otimes}1)\wh{\otimes}1\circ\Psi^{-1}=\sum_{\chi\in\Char(G)}\sum_{j=1}^{n}\eta\big(\,\!_{\Gamma(\Cl(G))^G}(\chi^{-1}\D_G(\chi x_j)\wh{\otimes}\bc|x_j\wh{\otimes}1)\big)P_{\chi\zeta}.
$$
\end{lemma}
\begin{proof}
Let $\bc$ be the generator of $\bC$, let 
$\xi\in\dom(\Psi\circ(\D_1\wh{\otimes}1)\wh{\otimes}1\circ\Psi^{-1})$, 
and choose sequences $(a_k)_{k=1}^\infty\subset A$ and 
$(\xi_k)_{k=1}^\infty\subset \h_\zeta$ such that $a_k\xi_k\rightarrow\xi$ 
as $k\rightarrow\infty$. Then
\begin{align*}
&\Psi\circ(\D_1\wh{\otimes}1)\wh{\otimes}1\circ\Psi^{-1}\xi\\
&=\Psi\sum_{\chi\in\Char(G)}\sum_{j=1}^{n}\lim_{k\rightarrow\infty}\lim_{\ell\rightarrow\infty}(-1)^{\deg a_k}\big(\big(\Phi_\chi(a_k)\wh{\otimes}(\D_G(\chi x_j)\wh{\otimes}\bc)\big)\wh{\otimes}(\phi_\ell\wh{\otimes}\ol{x_j\wh{\otimes}1})\big)\wh{\otimes}\xi_k\\
&=\sum_{\chi\in\Char(G)}\sum_{j=1}^{n}\lim_{k\rightarrow\infty}\eta\big(\,\!_{\Gamma(\Cl(G))^G}(\chi^{-1}\D_G(\chi x_j)\wh{\otimes}\bc|x_j\wh{\otimes}1)\big)\Phi_\chi(a_k)\xi_k\\
&=\sum_{\chi\in\Char(G)}\sum_{j=1}^{n}\eta\big(\,\!_{\Gamma(\Cl(G))^G}(\chi^{-1} \D_G(\chi^{-1}x_j)\wh{\otimes}\bc|x_j\wh{\otimes}1)\big)P_{\chi\zeta}\xi.&&\qedhere
\end{align*}
\end{proof}
\begin{prop}[The domain criterion]\label{prop:domain}
For all $\mu\in\mathbb{R}\take\{0\}$, the resolvent $(i\mu+\D)^{-1}$ maps the 
submodule $C_c^\infty(\Psi\circ(\D_1\wh{\otimes}1)\wh{\otimes}1\circ\Psi^{-1})\h$ 
into $\dom(\Psi\circ(\D_1\wh{\otimes}1)\wh{\otimes}1\circ\Psi^{-1})$.
\end{prop}
\begin{proof}
By Lemma \ref{lem:d111} and the compactness of $(1+\D_G)^{-1/2}$, 
if $\xi\in C_c^\infty(\Psi\circ(\D_1\wh{\otimes}1)\wh{\otimes}1\circ\Psi^{-1})\h$, 
then $P_\chi\xi=0$ for all but finitely many $\chi\in\Char(G)$. Since $(i\mu+\D)^{-1}$ 
commutes with the action of $G$, it preserves $\h_\chi$ for all $\chi\in\Char(G)$. 
Hence if $\xi\in C_c^\infty(\Psi\circ(\D_1\wh{\otimes}1)\wh{\otimes}1\circ\Psi^{-1})\h$, 
then $P_\chi(i\mu+\D)^{-1}\xi=0$ for all but finitely many $\chi\in\Char(G)$. Lemma \ref{lem:d111} 
then implies that $(i\mu+\D)^{-1}\in\dom(\Psi\circ(\D_1\wh{\otimes}1)\wh{\otimes}1\circ\Psi^{-1})$. \end{proof}
Since the connection and domain criteria of \cite{Kucerovsky}*{Theorem 13} are satisfied 
(Propositions \ref{prop:connection} and \ref{prop:domain} respectively), 
Theorem \ref{thm:main} is proved by combining the remaining positivity condition with Lemma \ref{lem:d111}.
\section{Factorisation for an odd spectral triple.}\label{sect:oddtriple}
Recall that $G$ is a compact abelian Lie group, equipped with the normalised Haar measure and a trivial spinor bundle $\bs{\$}_G$. However, suppose that rather than an \emph{even} $G$-equivariant spectral triple, we instead have an \emph{odd} $G$-equivariant spectral triple $(\A,\h,\D)$.

The $K$-homology class of an odd spectral triple is defined by associating to it an even spectral triple.  
Let $\gamma=\left(\begin{smallmatrix}1&0\\0&-1\end{smallmatrix}\right)\in B(\C^2)$, 
and equip $\C^2$ with the $\Z_2$-grading defined by $\gamma$.
Let $c$ be the generator of the Clifford algebra $\Cl_1$, and define a 
$\Z_2$-graded $\ast$-homomorphism $\Cl_1\rightarrow B(\C^2)$ by 
$c\mapsto\left(\begin{smallmatrix}0&1\\1&0\end{smallmatrix}\right)$. 
Equip $A\wh{\otimes}\Cl_1$ and $\h\wh{\otimes}\C^2$ with the obvious actions by $G$. 
Let $\omega=\left(\begin{smallmatrix}0&-i\\i&0\end{smallmatrix}\right)\in B(\C^2)$. 
Then $(\A\wh{\otimes}\Cl_1,\h\wh{\otimes}\C^2,\D\wh{\otimes}\omega)$ is an even 
$G$-equivariant spectral triple. The class of $(\A,\h,\D)$ in odd $K$-homology is defined to be 
$[(\A\wh{\otimes}\Cl_1,\h\wh{\otimes}\C^2,\D\wh{\otimes}\omega)]\in KK_G(A\wh{\otimes}\Cl_1,\C)=KK^1_G(A,\C)$, 
\cite{Connes}*{Prop. IV.A.13}.

We make the following definition analogously to Definition \ref{defn:theta}.
\begin{defn}\label{defn:thetaodd}
Let $(\A,\h,\D)$ be an odd, $G$-equivariant spectral triple for a trivially $\Z_2$-graded separable $C^*$-algebra $A$, and 
suppose that $\A$ is $\alpha$-compatible. Let $\zeta\in\Char(G)$ satisfy $\ol{A\h_\zeta}=\h$, 
and let $\eta:\Gamma(\Cl(G))^G\rightarrow B(\h)$ be a unital, 
equivariant $\ast$-homomorphism such that\\
1) $[\eta(s),a]=0$ for all $s\in\Gamma(\Cl(G))^G$ and $a\in A^G$, and\\
2) $a\eta(s)\cdot\dom(\D_\zeta)\subset\dom(\D)$ and 
$(\D \eta(s)-(-1)^{\deg s}\eta(s)\D) aP_\zeta$ is bounded on 
$\h$ for every $a\in \oplus_\chi\A_\chi$, $s\in\Gamma(\Cl(G))^G$, where 
$P_\zeta\in B(\h)$ is the orthogonal projection onto $\h_\zeta$. 

We define a $\Z_2$-graded $\ast$-homomorphism 
$\wt{\eta}:\Gamma(\Cl(G))^G\rightarrow B(\h\wh{\otimes}\C^2)$
 by $\wt{\eta}(s)=\eta(s)\wh{\otimes}\omega^{\deg s}$, where
$(\eta(s)\wh{\otimes}\omega^{\deg s})(\xi\wh{\otimes}v)=\eta(s)\xi\wh{\otimes}\omega^{\deg s}v.$
\end{defn}
It is easy to see that the pair $(\zeta,\wt{\eta})$ satisfy 
the conditions of Definition \ref{defn:theta} for the even 
$G$-equivariant spectral triple $(\A\wh{\otimes}\Cl_1,\h\wh{\otimes}\C^2,\D\wh{\otimes}\omega)$.

The next result follows easily from Theorem \ref{thm:main} 
applied to the even $G$-equivariant spectral triple 
$(\A\wh{\otimes}\Cl_1,\h\wh{\otimes}\C^2,\D\wh{\otimes}\omega)$.
\begin{thm}
\label{thm:odd}
Let $(\A,\h,\D)$ be an odd, $G$-equivariant spectral 
triple for a trivially $\Z_2$-graded 
$C^*$-algebra $A$, and let $\zeta\in\Char(G)$ and 
$\eta:\Gamma(\Cl(G))^G\rightarrow B(\h)$ be as in 
Definition \ref{defn:thetaodd}, so in particular $\A$ is $\alpha$-compatible. Let $(x_j)_{j=1}^{n}$ 
be a $G$-invariant global orthonormal frame for 
$\bs{\$}_G$. If there is some $R\in\mathbb{R}$ such that
\begin{align*}
&\sum_{j=1}^{n}\Big(\Ideal{\D\xi,\eta\big(\,\!_{\Gamma(\Cl(G))^G}(\chi^{-1}\D_G(\chi x_j)
\wh{\otimes}\bc|x_j\wh{\otimes}1)\big)P_{\chi\zeta}\xi}\\
&\qquad+\Ideal{\eta\big(\,\!_{\Gamma(\Cl(G))^G}(\chi^{-1}\D_G(\chi x_j)\wh{\otimes}\bc|x_j
\wh{\otimes}1)\big)P_{\chi\zeta}\xi,\D\xi}\Big)\geq R\|\xi\|^2
\end{align*}
for all $\chi\in\Char(G)$, $\xi\in\dom(\D)$, then the odd spectral triple $(\A,\h,\D)$ 
represents the Kasparov product of the left-hand, middle and 
right-hand modules for $(\A\wh{\otimes}\Cl_1,\h\wh{\otimes}\C^2,\D\wh{\otimes}\omega)$.
\end{thm}
\section{The $\theta$-deformation of a $\mathbb{T}^n$-equivariant spectral triple and factorisation.}
Given a $\mathbb{T}^n$-equivariant spectral triple 
$(\A,\h,\D)$ and a skew-symmetric matrix $\theta\in M_n(\R)$, 
one can construct the $\theta$-deformed $\mathbb{T}^n$-equivariant spectral triple 
$(\A_\theta,\h_\theta,\D_\theta)$. We show that if 
factorisation is achieved for $(\A,\h,\D)$, then it is also achieved for $(\A_\theta,\h_\theta,\D_\theta)$.

We first recall the construction of a $\theta$-deformed 
$\mathbb{T}^n$-equivariant spectral triple, \cites{ConnesLandi,Rieffel}. 
\begin{defn}
Let $\theta\in M_n(\R)$ be a skew-symmetric matrix. 
The \emph{noncommutative torus} $C(\mathbb{T}^n)_\theta$ 
is the universal $C^*$-algebra generated by $n$ unitaries $U_1,\ldots,U_n$ 
subject to the commutation relations $U_jU_k=e^{2\pi i\theta_{jk}}U_kU_j$ for $j,k=1,\ldots,n$. 
\end{defn}
The noncommutative torus $C(\mathbb{T}^n)_\theta$ carries an action 
by the $n$-torus $\mathbb{T}^n$, which is given by $t\cdot U_j=e^{2\pi it^j}U_j$, where 
$t=(t^1,\ldots,t^n)\in\mathbb{T}^n$ are the standard torus coordinates.
\begin{defn}
Let $A$ be a $\Z_2$-graded $C^*$-algebra with an action 
$\alpha$ by $\mathbb{T}^n$. 
Let $\theta\in M_n(\R)$ be a skew-symmetric matrix, 
and equip the tensor product $A\wh{\otimes}C(\mathbb{T}^n)_\theta$ 
with the diagonal action 
$t\cdot (a\wh{\otimes}b)=\alpha_t(a)\wh{\otimes}(t\cdot b)$ by 
$\mathbb{T}^n$. The \emph{$\theta$-deformation} of $A$ is the invariant sub-$C^*$-algebra 
$A_\theta:=(A\wh{\otimes}C(\mathbb{T}^n)_\theta)^{\mathbb{T}^n}$.

The $\theta$-deformation $A_\theta$ carries an action $\alpha^{(\theta)}$ by $\mathbb{T}^n$, 
given by $\alpha^{(\theta)}_t(a\wh{\otimes} b)=\alpha_t(a)\wh{\otimes}b.$ 
\end{defn}
\begin{defn}
Let $\h=\h^0\oplus\h^1$ be a $\Z_2$-graded Hilbert space with a strongly continuous 
unitary representation $V:\mathbb{T}^n\rightarrow U(\h)$ such that $V_t\cdot\h^j\subset\h^j$
 for $t\in\mathbb{T}^n$, $j\in\Z_2$. Let $\theta\in M_n(\R)$ be a skew-symmetric matrix. 
 Viewing $C(\mathbb{T}^n)_\theta$ as a right Hilbert module over itself, form the 
 $\Z_2$-graded right Hilbert $C(\mathbb{T}^n)_\theta$-module 
 $\h\wh{\otimes}C(\mathbb{T}^n)_\theta$. This module carries an action by 
 $\mathbb{T}^n$, given by
  $t\cdot (\xi\wh{\otimes}b)=V_t\xi\wh{\otimes}(t\cdot b).$ 
The \emph{$\theta$-deformation} of $\h$ is the $\Z_2$-graded Hilbert space 
$\h_\theta:=(\h\wh{\otimes}C(\mathbb{T}^n)_\theta)^{\mathbb{T}^n}.$ 
We define a unitary represenation 
$V^{(\theta)}:\mathbb{T}^n\rightarrow U(\h_\theta)$ by 
$V^{(\theta)}_t(\xi\wh{\otimes}b)=V_t\xi\wh{\otimes}b.$
\end{defn}
We can now define the $\theta$-deformed $\mathbb{T}^n$-equivariant 
spectral triple $(\A_\theta,\h_\theta,\D_\theta)$.
\begin{defn}
Suppose that $\A$ is $\alpha$-compatible.
Let $(\A,\h,\D)$ be a $\mathbb{T}^n$-equivariant spectral triple, and let $\theta\in M_n(\R)$ be skew-symmetric. Represent $A_\theta$ on $\h_\theta$ by 
$(a\wh{\otimes} b)(\xi\wh{\otimes}c)=a\xi\wh{\otimes}bc$ (for $a\in A$, $b\in C(\mathbb{T}^n)_\theta$), 
and setting $U^k:=U_1^{k_1}\cdots U_n^{k_n}$ for $k\in\Z^n$, let 
$$
\A_\theta=\Span\{a_k\wh{\otimes}U^{-k}\in A_\theta:a_k\in\A\cap A_k,\ k\in\Z^n\}
$$
which is a dense sub-$\ast$-algebra of $A_\theta$ compatible with $\alpha^{(\theta)}$, 
and define an operator 
$\D_\theta$ on $\h_\theta$ by $\D_\theta(\xi\wh{\otimes}b)=\D\xi\wh{\otimes}b$ for $\xi\in\dom(\D).$ 
Then $(\A_\theta,\h_\theta,\D_\theta)$ is a $\mathbb{T}^n$-equivariant spectral triple for $A_\theta$, 
which we call the \emph{$\theta$-deformation} of $(\A,\h,\D)$.
\end{defn}
\begin{prop}
Let $A$ be a $C^*$-algebra with an action by $\mathbb{T}^n$, and let $\theta\in M_n(\R)$ be skew-symmetric. Then $A_\theta$ satisfies the spectral subspace assumption if and only if $A$ does.
\end{prop}
\begin{proof}
Let $\psi:A^{\mathbb{T}^n}\rightarrow A_\theta^{\mathbb{T}^n}$ be the $\ast$-isomorphism
 $\psi(a)=a\wh{\otimes}1$. Then $\psi(\ol{A_kA_k^*})=\ol{(A_\theta)_k(A_\theta)_k^*}$ 
 for all $k\in\Z^n$. 
 \end{proof}
\begin{defn}
Define a unitary isomorphism $u:\h\rightarrow\h_\theta$ 
by $u\left(\sum_{k\in\Z^n}\xi_k\right)=\sum_{k\in\Z^n}\xi_k\wh{\otimes}U^{-k}$. 
This isomorphism intertwines the actions of $\mathbb{T}^n$, 
so that $u:\h_\ell\rightarrow(\h_\theta)_\ell$ for all $\ell\in\Z^n$.

Given $\eta:\Gamma(\Cl(\mathbb{T}^n))^{\mathbb{T}^n}\rightarrow B(\h)$, define 
$\eta_\theta:\Gamma(\Cl(\mathbb{T}^n))^{\mathbb{T}^n}\rightarrow B(\h_\theta)$ by
$\eta_\theta(s)=u\circ\eta(s)\circ u^*$.
\end{defn}
\begin{prop}
The pair $(\ell,\eta_\theta)$ satisfies the conditions of 
Definition \ref{defn:theta} for $(\A_\theta,\h_\theta,\D_\theta)$ if and only if 
$(\ell,\eta)$ satisfies those conditions for $(\A,\h,\D)$. Consequently 
$(\A_\theta,\h_\theta,\D_\theta)$ factorises if and only if $(\A,\h,\D)$ does.
\end{prop}
\begin{proof}
If $\xi\wh{\otimes}U^{-\ell}\in(\h_\theta)_\ell$ and $a\wh{\otimes}U^{-k}\in(A_\theta)_k$, then 
$(a\wh{\otimes}U^{-k})(\xi\wh{\otimes}U^{-\ell})=\lambda a\xi\wh{\otimes} U^{-k-\ell}$
for some  $\lambda\in U(1)$. Hence $\ol{A_\theta(\h_\theta)_\ell}=\h$ if and only if $\ol{A\h_\ell}=\h$. 

Recall the $\ast$-isomorphism $\psi:A^{\mathbb{T}^n}\rightarrow A_\theta^{\mathbb{T}^n}$, 
$\psi(a)=a\wh{\otimes}1$. Then $u(a\xi)=\psi(a)u(\xi)$ for all $a\in A^{\mathbb{T}^n}$, $\xi\in\h$. 
Hence $u\circ [\eta(s),a]_\pm\circ u^*=[\eta_\theta(s),\psi(a)]_\pm$
for all $s\in\Gamma(\Cl(\mathbb{T}^n))^{\mathbb{T}^n}$, $a\in A^{\mathbb{T}^n}$, so 
Condition (1) is satisfied for the $\theta$-deformation 
if and only if it is satisfied for the original spectral triple.

By construction, $\oplus_k(\A_\theta)_k=\A_\theta$. Let $a\wh{\otimes}U^{-k}\in (\A_\theta)_k$ 
and let $s\in\Gamma(\Cl(\mathbb{T}^n))^{\mathbb{T}^n}$. 
If $\xi\wh{\otimes}U^{-\ell}\in(\h_\theta)_\ell$ then 
$u^*\big((a\wh{\otimes} U^{-k})\eta_\theta(s)(\xi\wh{\otimes}U^{-\ell})\big)=\lambda a\eta(s)\xi$ 
for some $\lambda\in U(1)$. Since $\D_\theta=u\circ \D\circ u^*$, 
it follows that $a\eta(s)\cdot\dom(\D_\ell)\subset\dom(\D)$ for all 
$a\in\oplus_k\A_k$, $s\in\Gamma(\Cl(\mathbb{T}^n))^{\mathbb{T}^n}$ if and only if 
$b\eta_\theta(s)\cdot\dom((\D_\theta)_\ell)\subset\dom(\D_\theta)$ for all 
$b\in\A_\theta$, $s\in\Gamma(\Cl(\mathbb{T}^n))^{\mathbb{T}^n}$.

Let $a\wh{\otimes}U^{-k}\in(\A_\theta)_k$, and let 
$s\in\Gamma(\Cl(\mathbb{T}^n))^{\mathbb{T}^n}$. Then 
$$
u^*\circ[\D_\theta,\eta_\theta(s)]_\pm(a\wh{\otimes}U^{-k}) P_\ell\circ u=\lambda [\D,\eta(s)]_\pm a P_\ell
$$ 
for some $\lambda\in U(1)$ depending on $k$, $\ell$ and $\theta$. 
Therefore $(\ell,\eta)$ satisfies Condition (2) if and only if 
$(\ell,\eta_\theta)$ satisfies Condition (2).

Since $\D_\theta=u\circ \D\circ u^*$ and $\eta_\theta=u\circ \eta\circ u^*$, clearly the factorisation criterion (Theorems \ref{thm:main}, \ref{thm:odd}) is satisfied for $(\ell,\eta_\theta)$ and the $\theta$-deformed spectral triple if and only if it is satisfied for $(\ell,\eta)$ and the original spectral triple.
\end{proof}
\section{Factorisation of a torus-equivariant Dirac-type operator over a compact manifold.}\label{sect:manifoldfact}
Throughout this section, let $(M,g)$ be a compact 
Riemannian manifold with a smooth, free, isometric 
left action by the $n$-torus $\mathbb{T}^n$, and let 
$S$ be a (possibly $\Z_2$-graded) 
$\mathbb{T}^n$-equivariant Clifford module over 
$M$ equipped with a $\mathbb{T}^n$-invariant 
Clifford connection $\nabla^S$, \cite{BerlineGetzlerVergne}*{p.186}. 
Then $(C^\infty(M),L^2(S),\D)$ is a $\mathbb{T}^n$-equivariant 
spectral triple, where $\D$ is the associated Dirac operator on $S$. 
The spectral triple is even if $S$ is $\Z_2$-graded; otherwise it is odd.

We will show that $(C^\infty(M),L^2(S),\D)$ can always be factorised. 
We show that each  condition (the spectral subspace assumption, 
the map $\eta:\Gamma(\Cl(\mathbb{T}^n))^{\mathbb{T}^n}\rightarrow B(L^2(S))$ 
and the positivity criterion) is satisfied in turn. Compatibility of $C^\infty(M)$ with the
action is satisfied since we assume the action to be smooth.

We show that we in fact have full spectral subspaces, which is a special case of the SSA.
\begin{prop}\label{prop:fullSS}
Let $N$ be a manifold with a smooth free left action by the $n$-torus $\mathbb{T}^n$. Then 
$C_0(N)$ has full spectral subspaces; i.e. $\ol{C_0(N)_kC_0(N)_k^*}=C_0(N)^{\mathbb{T}^n}$ 
for all $k\in\Z^n$.
\end{prop}
\begin{proof}
The closed ideals of $C_0(N/\mathbb{T}^n)$ are in one-to-one correspondence 
with the closed subspaces of $N/\mathbb{T}^n$, where the ideal corresponding 
to $X\subset N/\mathbb{T}^n$ is $\{f\in C_0(N/\mathbb{T}^n):f|_X=0\}$. 
Since $\ol{C_0(N)_kC_0(N)_k^*}$ is an ideal in $C_0(N)^{\mathbb{T}^n}\cong C_0(N/\mathbb{T}^n)$, 
it is enough to show that for all $x\in N/\mathbb{T}^n$, there is some $a,b\in C_0(N)_k$ 
such that $ab^*|_{\pi^{-1}(\{x\})}\neq0$, where $\pi:N\rightarrow N/\mathbb{T}^n$ is the quotient map.

Since the action of $\mathbb{T}^n$ on $N$ is proper and free, $N$ is a principal 
$\mathbb{T}^n$-bundle over $N/\mathbb{T}^n$. Let $x\in N/\mathbb{T}^n$,
 and let $U$ be a neighbourhood of $x$ such that $\pi^{-1}(U)\cong U\times \mathbb{T}^n$ 
 as $\mathbb{T}^n$-spaces. Then there is an equivariant $\ast$-isomorphism 
 $C_0(\pi^{-1}(U))\cong C_0(U)\otimes C(\mathbb{T}^n)$. Under this 
 isomorphism, functions in $C_0(\pi^{-1}(U))_{-k}\subset C_0(N)_{-k}$ have 
 the form $a\otimes\chi_k$, where $\chi_k\in C(\mathbb{T}^n)$ is the character 
 $\chi_k(t)=e^{2\pi it\cdot k}$. Let 
 $a\otimes\chi_k$, $b\otimes\chi_k\in C_0(U)\otimes C(\mathbb{T}^n)_{-k}$ 
be functions such that $a(x),b(x)\neq0$. 
Then $(a\otimes\chi_k)(b\otimes\chi_k)^*=ab^*\otimes1$, 
 and $ab^*(x)\neq0$.
\end{proof}
We require a character $\ell\in\Z^n$ and a map 
$\eta:\Gamma(\Cl(\mathbb{T}^n))^{\mathbb{T}^n}\rightarrow B(L^2(S))$ 
satisfying the conditions of Definition \ref{defn:theta} 
(or Definition \ref{defn:thetaodd} if $S$ is trivially graded). 
The following lemma shows that any $\ell\in\Z^n$ satisfies the condition 
(and indeed factorisation is achieved for any choice of $\ell$).
\begin{lemma}
\label{lem:anyellworks}
Let $N$ be a Riemannian manifold 
with a smooth free left action by the $n$-torus $\mathbb{T}^n$, 
and let $F$ be an equivariant Hermitian vector bundle over $N$. 
Then $\ol{C_0(N)L^2(F)_\ell}=L^2(F)$ for all $\ell\in\Z^n$. 
\end{lemma}
\begin{proof}
Since $L^2(F)=\bigoplus_{k\in\Z^n}L^2(F)_k$, it is enough to show that 
$C_0(N)_{k-\ell}L^2(F)_\ell$ is dense in $L^2(F)_k$ for all $k\in\Z^n$. 
We show that $C_0(N)_{k-\ell}\Gamma_c(F)_{\ell}=\Gamma_c(F)_k$ 
for all $k\in\Z$, which since $\Gamma_c(F)$ is dense in $L^2(F)$ proves the result.

Let $\xi\in \Gamma_c(F)_k$. Since $\xi$ has compact support, 
there is a finite collection of open sets $(U_i)_{i=1}^N$ which cover 
the support of $\xi$, such that $U_i\cong\pi(U_i)\times\mathbb{T}^n$ 
as $\mathbb{T}^n$-spaces, recalling the quotient map $\pi:N\rightarrow N/\mathbb{T}^n$. 
Let $(\phi_n)_{n=1}^N$ be an invariant 
paritition of unity for $\bigcup_{i=1}^NU_i$ subordinate to $(U_i)_{i=1}^N$. 
For each $i=1,\ldots,N$, let $f_i\in C_0(\pi(U_i))$ be a function such that 
$(f_i\circ \pi)\phi_i=f_i\circ\pi$, and let $a_i,b_i\in C_0(U_i)$ be the functions 
corresponding to $f_i\otimes\chi_{k-\ell}$ and $f_i\otimes\chi_{\ell-k}$ respectively 
under the equivariant $\ast$-isomorphism $C_0(U_i)\cong C_0(\pi(U_i))\otimes C(\mathbb{T}^n)$. 
Note that $b_ia_i\phi_i=\phi_i$ and $a_i\xi\in \Gamma_c(F)_\ell$, so
$\xi=\sum_{i=1}^N\phi_i\xi=\sum_{i=1}^Nb_ia_i\phi_i\xi_i\in C_0(N)_{k-\ell}\Gamma_c(F)_\ell.$ 
\end{proof}
We will assume that $\ell\in\Z^n$ is fixed from now on. 
This choice does not affect the factorisation. This means we could choose 
$\ell=0$ for convenience, but we will leave $\ell$ arbitrary in order to 
show that factorisation is achieved for all choices of $\ell$.

Next we define the map $\eta:\Gamma(\Cl(\mathbb{T}^n))^{\mathbb{T}^n}\rightarrow B(L^2(S))$.
First recall that the fundamental vector field 
$X^{(v)}\in \Gamma^\infty(TM)$ associated to $v\in T_e\mathbb{T}^n$ is 
$X^{(v)}_x=\left.\frac{d}{dt}\exp(tv)\cdot x\right|_{t=0}$. Since the action of the $n$-torus 
$\mathbb{T}^n$ on $M$ is free, the fundamental vector field of 
a non-zero vector in $T_e\mathbb{T}^n$ is non-vanishing.
The canonical isomorphisms 
$T_e\mathbb{T}^n\cong\Gamma(T^*\mathbb{T}^n)^{\mathbb{T}^n}$ 
and  $TM\cong T^*M$, along with the 
fundamental vector field map, give us an equivariant, $\Z_2$-graded map 
$
\Gamma(T^*\mathbb{T}^n)^{\mathbb{T}^n}\rightarrow \Gamma^\infty(T^*M).
$
However, this map need not be an isometry 
and hence need not extend to a $\ast$-homomorphism 
$\Gamma(\Cl(\mathbb{T}^n))^{\mathbb{T}^n}\rightarrow\Gamma^\infty(\Cl(M))$. 
We will modify this map to obtain a $*$-homomorphism. 
For $j=1,\ldots,n$, let $X_j\in\Gamma^\infty(TM)^{\mathbb{T}^n}$ be the 
fundamental vector field associated to $\frac{\de}{\de t^j}\in T_e\mathbb{T}^n$. 
Observe that $\{X_1(x),\ldots,X_n(x)\}$ is a linearly independent set for every $x\in M$. 
For each $x\in M$, let $W(x)=(W^{jk}(x))_{j,k=1}^n\in M_n(\R)$ be the inverse 
square root of the positive-definite matrix $(g(X_j(x),X_k(x)))_{j,k=1}^n$. 
Letting $x$ vary, we obtain functions $W^{jk}\in C^\infty(M)^{\mathbb{T}^n}$ 
for $j,k=1,\ldots,n$. Let
\begin{align}\label{eq:normalisedcovectors}
v_k=\sum_{j=1}^nX^\flat_jW^{jk}\in\Gamma^\infty(T^*M)^{\mathbb{T}^n},\quad k=1,\ldots,n,
\end{align} 
where $TM\rightarrow T^*M$, $X\mapsto X^\flat$ is the canonical isomorphism. 
Then $\{v_1(x),\ldots,v_n(x)\}$ is an orthonormal set for all $x\in M$. 
We call the functions $W^{jk}\in C^\infty(M)^{\mathbb{T}^n}$, $j,k=1,\ldots,n$ 
the \emph{normalisation functions}.
\begin{defn}
\label{def:ee-ta}
The map $\Gamma(T^*\mathbb{T}^n)^{\mathbb{T}^n}\ni dt^k\mapsto -v_k=
-\sum_{j=1}^nX^\flat_jW^{jk}\in\Gamma^\infty(T^*M)^{\mathbb{T}^n}$ 
is now not only equivariant and $\Z_2$-graded (when $S$ is $\Z_2$-graded), 
but is also an isometry. It therefore extends to a unital 
$\ast$-homomorphism 
$\eta:\Gamma(\Cl(\mathbb{T}^n))^{\mathbb{T}^n}\rightarrow \Gamma^\infty(\Cl(M))\subset B(L^2(S))$.
\end{defn}
\begin{rem}
The appearance of a minus sign in the definition of 
$\eta$ arises as follows.  The torus action on sections of the Clifford module $S$ is 
$V_{\exp(tv)}u(x)=\exp(tv)\cdot u(\exp(-tv)\cdot x)$. 
So the more natural convention  to define $\eta$ is to use the vector field
$Y^{(v)}_x=\left.\frac{d}{dt}\exp(-tv)\cdot x\right|_{t=0}=-X^{(v)}_x$. 
\end{rem}
As functions are central in the endomorphisms, 
$\eta$ satisfies Condition 1) of Definition \ref{defn:theta}, 
so it remains to check Condition 2).
Since the image of $\eta$ consists of smooth sections 
of $\Cl(M)$, $\eta(s)\cdot\dom(\D)\cap L^2(S)_\ell\subset\dom(\D)$ 
for all $s\in\Gamma(\Cl(\mathbb{T}^n))^{\mathbb{T}^n}$. 
Before showing that $[\D,\eta(s)]_\pm P_\ell$ is bounded for all 
$s\in\Gamma(\Cl(\mathbb{T}^n))^{\mathbb{T}^n}$, we prove a lemma.
\begin{lemma}\label{lem:infgen}
Let $N$ be a  
Riemannian manifold, and let $G$ be a Lie group acting 
smoothly by isometries on $N$. Let $F$ be an equivariant Hermitian 
vector bundle over $N$. This defines a 
unitary representation $V:G\rightarrow U(L^2(F))$.

 Let $v\in\mathfrak{g}$, and let $X^{(v)}\in\Gamma^\infty(TN)$ 
 be the fundamental vector field associated to $v$. Define a 
 one-parameter unitary group on $L^2(F)$ by $\gamma_v(t)=V_{\exp(tv)}$. 
 Let $A$ be the infinitesimal generator of $\gamma_v$, characterised by 
 $\gamma_v(t)=e^{itA}$. Then\\
1) $A:\Gamma^\infty(F)\rightarrow\Gamma^\infty(F)$, and\\
2) $iA+\nabla_{X^{(v)}}\in\Gamma^\infty(\End(F))$ for any connection $\nabla$ on $F$.

In particular, if $N$ is compact, then $iA+\nabla_{X^{(v)}}\in B(L^2(F))$ for any connection $\nabla$. 
\end{lemma}
\begin{proof}
Let $u\in\Gamma^\infty(F)$. Working on a local trivialisation of $F$, 
we can view $u$ as a $\C^k$-valued function on $N$. Since 
$\gamma_v(t)u(x)=\exp(tv)\cdot u(\exp(-tv)\cdot x)$, in this trivialisation,
$$
iAu(x)=\frac{d}{dt}\gamma_v(t)u(x)\Big|_{t=0}=Bu(x)-X^{(v)}_x(u),
$$
where $B\in M_k(\C)$ is the derivative at $t=0$ of the curve 
$t\mapsto\exp(tv)\in M_k(\C)$. 
This shows 1) 
and 2), 
since if $\nabla$ is a connection then locally $\nabla_{X^{(v)}}=X^{(v)}+\omega$, 
where $\omega$ is a locally-defined $M_k(\C)$-valued function on $N$.
\end{proof}
The next result shows that the pair $(\ell,\eta)$ satisfy the remaining condition (2) of 
Definition \ref{defn:theta}. 
\begin{prop} Let $\eta$ be as in Definition \ref{def:ee-ta} and $\ell\in\Z^n$. 
Then the graded commutator
$[\D,\eta(s)]_\pm P_\ell$ is bounded for all $s\in\Gamma(\Cl(\mathbb{T}^n))^{\mathbb{T}^n}$.
\end{prop}
\begin{proof}
For $j=1,\ldots,n$, let $X_j$ be the fundamental vector field associated to $\frac{\de}{\de t^j}$, 
and let $v_j=\sum_{k=1}^nX_kW^{kj}$ be the normalised vector field as in Equation \eqref{eq:normalisedcovectors}. 
Let $U\subset M$ 
be an open set such that $M|_U$ is parallelisable, and choose vector fields 
$(w_1,\ldots,w_{m-n})\subset\Gamma^\infty(TU)$ (where $m:=\dim M$) 
such that $(v_1,\ldots,v_n,w_1,\ldots,w_{m-n})$ is an orthonormal frame for $TU$. 
We can locally express the Dirac operator $\D$ as 
$$
\D|_U=\sum_{j=1}^nc(v_j^\flat)\nabla^S_{v_j}+\sum_{i=1}^{m-n}c(w_i^\flat)\nabla^S_{w_i},
$$
where $v\mapsto v^\flat$ is the 
isomorphism $TM\rightarrow T^*M$ determined by the Riemannian metric, 
and $c$ denotes Clifford multiplication. 

Since $\Gamma(\Cl(\mathbb{T}^n))^{\mathbb{T}^n}$ is generated by 
$(c(dt^k))_{k=1}^n$, we need only show that the anticommutator 
$\{\D,c(v_j^\flat)\}P_\ell$ is bounded for $j=1,\ldots,n$. 
Letting $\nabla^{LC}$ be the Levi-Civita connection on $T^*M$ and using the
compatibility between $\nabla^S$ and $\nabla^{LC}$, we have
\begin{align*}
\{\D,c(v_j^\flat)\}|_U
&=\sum_{i=1}^nc(v_i^\flat)c(v_j^\flat)\nabla^S_{v_i}
+\sum_{i=1}^mc(w_i^\flat)c(v_j^\flat)\nabla^S_{w_i}
+\sum_{i=1}^nc(v_i^\flat)c(\nabla^{LC}_{v_i}v_j^\flat)\\
&\quad
+\sum_{i=1}^mc(w_i^\flat)c(\nabla^{LC}_{w_i}v_j^\flat)
+\sum_{i=1}^nc(v_j^\flat)c(v_i^\flat)\nabla^S_{v_i}
+\sum_{i=1}^mc(v_j^\flat)c(w_i^\flat)\nabla^S_{w_i}\\
&=-2\nabla^S_{v_j}
+\sum_{i=1}^nc(v_i^\flat)c(\nabla^{LC}_{v_i}v_j^\flat)
+\sum_{i=1}^mc(w_i^\flat)c(\nabla^{LC}_{w_i}v_j^\flat).
\end{align*}
The second and third terms are smooth endomorphisms which 
are independent of the choice of $(f_1,\ldots,f_{m-n})$,
and so globally 
$$
\{\D,c(v_j^\flat)\}=-2\nabla^S_{v_j}+\text{bundle endomorphism}
=-2\sum_{k=1}^nW^{kj}\nabla_{X_k}^S+\text{bundle endomorphism.}
$$
Since $M$ is compact, every endomorphism is bounded,  
and so it is enough to show that $\nabla^S_{X_j}P_\ell$ is bounded. 
By Lemma \ref{lem:infgen}, $\nabla^S_{X_j}=-iA_j+\omega$ for some 
$\omega\in\Gamma^\infty(\End(S))$, where $A_j$ is the infinitesimal generator 
of the one-parameter unitary group $s\mapsto V_{\exp(s\frac{\de}{\de t^j})}\in U(L^2(S))$. Since 
$$
\exp(s\frac{\de}{\de t^j})=(0,\ldots,0,\underbrace{s}_{j\Th},0,\ldots,0),\qquad s\in\R,
$$
$V_{\exp(s\frac{\de}{\de t^j})}=\sum_{k\in\Z^n}e^{2\pi isk_j}P_k$. 
Hence $A_j=\sum_{k\in\Z^n}2\pi k_jP_k$, and thus
$$
\nabla^S_{X_j}P_\ell=-iA_jP_\ell+\omega P_{\ell}=-2\pi i\ell_jP_{\ell}+\omega P_\ell
$$
is bounded, and so we have shown that $\{\D,c(v_j^\flat)\}P_\ell$ is bounded.
\end{proof}
Now that we have a pair $(\ell,\eta)$ satisfying the conditions of Definition 
\ref{defn:theta}, it remains to check the positivity criterion. To this end we 
derive an explicit formula for $\Psi\circ(\D_1\wh{\otimes}1)\wh{\otimes}1\circ\Psi^{-1}$, 
recalling from Equation \eqref{eq:sigh} the isomorphism 
$$
\Psi:(E_1\wh{\otimes}_{C(M)^{\mathbb{T}^n}\wh{\otimes}\bC}(C(M)^{\mathbb{T}^n}\wh{\otimes}(\Gamma(\bs{\$}_{\mathbb{T}^n})^{\mathbb{T}^n}\wh{\otimes}\bC)^*))\wh{\otimes}_{C(M)^{\mathbb{T}^n}\wh{\otimes}\Gamma(\Cl(\mathbb{T}^n))^{\mathbb{T}^n}}L^2(S)_\ell\rightarrow L^2(S).
$$
\begin{lemma}
\label{lem:d1otimes1}
For $j=1,\ldots,n$, let $X_j\in\Gamma^\infty(TM)$ be the fundamental vector field 
associated to $\frac{\de}{\de t^j}\in T_e\mathbb{T}^n$, with corresponding covector field 
$X_j^\flat$, and let $A_j$ be the infinitesimal generator of the one-parameter unitary group 
$t\mapsto V_{\exp(t\frac{\de}{\de t^j})}\in U(L^2(S))$.  
Let $W^{jk}\in C^\infty(M)^{\mathbb{T}^n}$ be the normalisation functions. 
Then
$$
\Psi\circ(\D_1\wh{\otimes}1)\wh{\otimes}1\circ\Psi^{-1}=-i\sum_{j,r=1}^nW^{rj}c(X_r^\flat)(A_j-2\pi \ell_j).
$$
\end{lemma}
\begin{proof}
Let $(x_r)_{r=1}^{2^{\floor{n/2}}}$ be an invariant, global orthonormal frame for 
$\bs{\$}_{\mathbb{T}^n}$, corresponding to some orthonormal basis for $(\bs{\$}_{\mathbb{T}^n})_e$. 
By Lemma \ref{lem:d111},
\begin{align*}
\Psi\circ(\D_1\wh{\otimes}1)\wh{\otimes}1\circ\Psi^{-1}&=\sum_{k\in\Z^n}\sum_{r=1}^{2^{\floor{n/2}}}\eta\big(\,\!_{\Gamma(\Cl(\mathbb{T}{^n}))^{\mathbb{T}^n}}(\chi_k^{-1}\D_{\mathbb{T}^n}(\chi_kx_r)\wh{\otimes}\bc|x_r\wh{\otimes}1)\big)P_{k+\ell}.
\end{align*}
Since we are using the trivial flat spinor bundle over $\mathbb{T}^n$, 
$\D_{\mathbb{T}^n}x_r=0$ for all $r$, and 
$$
[\D_{\mathbb{T}^n},\chi_k]=2\pi i\sum_{j=1}^nk_j\chi_kc(dt^j).
$$
Recall that $\eta:\Gamma(\Cl(\mathbb{T}^n))^{\mathbb{T}^n}\rightarrow B(L^2(S))$ 
is defined by $c(dt^j)\mapsto -\sum_{r=1}^nc(X_r^\flat)W^{rj}$. Hence
\begin{align*}
\Psi\circ(\D_1\wh{\otimes}1)\wh{\otimes}1\circ\Psi^{-1}
&=2\pi i\sum_{k\in\Z^n}\sum_{r=1}^{2^{\floor{n/2}}}\sum_{j=1}^nk_j
\eta\big(\,\!_{\Gamma(\Cl(\mathbb{T}{^n}))^{\mathbb{T}^n}}(c(dt^j)x_r\wh{\otimes}\bc|x_r\wh{\otimes}1)\big)P_{k+\ell}\\
&=2\pi i\sum_{k\in\Z^n}\sum_{r=1}^{2^{\floor{n/2}}}\sum_{j=1}^nk_j
\eta(c(dt^j))\eta\big(\,\!_{\Gamma(\Cl(\mathbb{T}{^n}))^{\mathbb{T}^n}}(x_r\wh{\otimes}1|x_r\wh{\otimes}1)\big)P_{k+\ell}\\
&=-2\pi i\sum_{k}\sum_{j,p=1}^nk_jW^{pj}c(X_p^\flat)P_{k+\ell}
=-i\sum_{j,r=1}^nW^{rj}c(X_r^\flat)(A_j-2\pi \ell_j).\qedhere
\end{align*}
\end{proof}
\begin{thm}
The positivity criterion is satisfied; that is there is some $R\in\mathbb{R}$ such that
$$
\Ideal{\D\xi,\Psi\circ(\D_1\wh{\otimes}1)\wh{\otimes}1\circ\Psi^{-1}\xi}+\Ideal{\Psi\circ(\D_1\wh{\otimes}1)\wh{\otimes}1\circ\Psi^{-1}\xi,\D\xi}\geq R\|\xi\|^2
$$
for all $\xi\in\dom(\D)\cap\Psi(\dom((\D_1\wh{\otimes}1)\wh{\otimes}1))$.
Thus $(C^\infty(M),L^2(S),\D)$ factorises.
\end{thm}
\begin{proof}
For $j=1,\ldots,n$, let $X_j\in\Gamma^\infty(TM)$ be the fundamental 
vector field corresponding to $\frac{\de}{\de t^j}\in T_e\mathbb{T}^n$, and let 
$v_j=\sum_{p=1}^nX_pW^{pj}$ be the normalised vector field 
as in Equation \eqref{eq:normalisedcovectors}. Let 
$U\subset M$ be an open set such that $M|_U$ is parallelisable, and 
choose vector fields $(w_1,\ldots,w_{m-n})\subset\Gamma^\infty(TU)$ 
(where $m:=\dim M$) such that $(v_1,\ldots,v_n,w_1,\ldots,w_{m-n})$ is 
an orthonormal frame for $TU$. Recall that we can locally express the 
Dirac operator $\D$ as 
$$
\D|_U=\sum_{j=1}^nc(v_j^\flat)\nabla^S_{v_j}+\sum_{i=1}^{m-n}c(w_i^\flat)\nabla^S_{w_i}.
$$
Since $M$ is compact, by using a partition of unity it is enough to prove 
the positivity for sections with support in an open set $V$ with $\ol{V}\subset U$.

Let $A_j$ be the generator of the one-parameter unitary group 
$s\mapsto V_{\exp(s\frac{\de}{\de t^j})}\in U(L^2(S))$ for $j=1,\ldots,n$. 
Then for $\xi\in\dom(\D)\cap\Psi(\dom((\D_1\wh{\otimes}1)\wh{\otimes}1))$ 
with support in $V$,
\begin{align*}
&\Ideal{\D\xi,\Psi\circ(\D_1\wh{\otimes}1)\wh{\otimes}1\circ\Psi^{-1}\xi}+\Ideal{\Psi\circ(\D_1\wh{\otimes}1)\wh{\otimes}1\circ\Psi^{-1}\xi,\D\xi}\\
&=\sum_{j,p}\Ideal{c(v_j^\flat)\nabla^S_{v_j}\xi,-ic(v_p^\flat)(A_p-2\pi\ell_p)\xi}+\sum_{j,p}\Ideal{c(w_j^\flat)\nabla^S_{w_j}\xi,-ic(v_p^\flat)(A_p-2\pi\ell_p)\xi}\\
&+\sum_{j,p}\Ideal{-ic(v_p^\flat)(A_p-2\pi\ell_p)\xi,c(v_j^\flat)\nabla^S_{v_j}\xi}+\sum_{j,p}\Ideal{-ic(v_p^\flat)(A_p-2\pi\ell_p)\xi,c(w_j^\flat)\nabla^S_{w_j}\xi}.
\end{align*}
Given $X\in\Gamma^\infty(TM)$, the (formal) adjoint of $\nabla_X$ is $(\nabla^S_X)^*=-\nabla^S_X-\dive X$. 
Using the compatibility between $\nabla^S$ and the Levi-Civita connection 
$\nabla^{LC}$ on $T^*M$, we compute
\begin{align*}
&\Ideal{\D\xi,\Psi\circ(\D_1\wh{\otimes}1)\wh{\otimes}1\circ\Psi^{-1}\xi}+\Ideal{\Psi\circ(\D_1\wh{\otimes}1)\wh{\otimes}1\circ\Psi^{-1}\xi,\D\xi}=4\pi i\sum(k_j-\ell_j)\Ideal{\xi,\nabla^S_{v_j}P_k\xi}\\
&-2\pi i\sum(k_p-\ell_p)\Ideal{\xi,\big(c(\nabla^{LC}_{v_j}v_j^\flat)c(v_p^\flat)+c(v_j^\flat)c(\nabla^{LC}_{v_j}v_p^\flat)+(\dive v_j)c(v_j^\flat)c(v_p^\flat)\big)P_k\xi}\\
&-2\pi i\sum(k_p-\ell_p)\Ideal{\xi,\big(c(\nabla^{LC}_{w_j}w_j^\flat)c(v_p^\flat)+c(w_j^\flat)c(\nabla^{LC}_{w_j}v_p^\flat)+(\dive w_j)c(w_j^\flat)c(v_p^\flat)\big)P_k\xi}.
\end{align*}
Let $\omega_j=\nabla^S_{X_j}+iA_j\in\Gamma^\infty(\End(S))$, as in Lemma \ref{lem:infgen}.  Since $A_jP_k=2\pi k_jP_k$, 
\begin{align*}
&\Ideal{\D\xi,\Psi\circ(\D_1\wh{\otimes}1)\wh{\otimes}1\circ\Psi^{-1}\xi}+\Ideal{\Psi\circ(\D_1\wh{\otimes}1)\wh{\otimes}1\circ\Psi^{-1}\xi,\D\xi}\\
&=8\pi^2\sum k_p(k_j-\ell_j)\Ideal{\xi,W^{jp}P_k\xi}+4\pi i\sum(k_j-\ell_j)\Ideal{\xi,W^{jp}\omega_pP_k\xi}\\
&\quad-2\pi i\sum(k_p-\ell_p)\Ideal{\xi,\big(c(\nabla^{LC}_{v_j}v_j^\flat)c(v_p^\flat)+c(v_j^\flat)c(\nabla^{LC}_{v_j}v_p^\flat)+(\dive v_j)c(v_j^\flat)c(v_p^\flat)\big)P_k\xi}\\
&\quad-2\pi i\sum(k_p-\ell_p)\Ideal{\xi,\big(c(\nabla^{LC}_{w_j}w_j^\flat)c(v_p^\flat)+c(w_j^\flat)c(\nabla^{LC}_{w_j}v_p^\flat)+(\dive w_j)c(w_j^\flat)c(v_p^\flat)\big)P_k\xi}.
\end{align*}
We estimate:
\begin{align*}
&\Ideal{\D\xi,\Psi\circ(\D_1\wh{\otimes}1)\wh{\otimes}1\circ\Psi^{-1}\xi}+\Ideal{\Psi\circ(\D_1\wh{\otimes}1)\wh{\otimes}1\circ\Psi^{-1}\xi,\D\xi}\\
&\geq8\pi^2\sum_{j,p,k}k_p(k_j-\ell_j)\Ideal{P_k\xi,W^{jp}P_k\xi}-\sum_{p,k}|k_p-\ell_p|C_p\Ideal{P_k\xi,P_k\xi},
\end{align*}
for some constants $C_p\in[0,\infty)$, $p=1,\ldots,n$, which are based 
on the norms of the endomorphisms such as $W^{jp}\omega_p$ and 
$(\dive w_j)c(w_j^\flat)c(v_p^\flat)$ on the compact set $\ol{V}$. 

For each $x\in M$, let $\lambda(x)>0$ be the smallest eigenvalue 
of the positive-definite real matrix $(W^{jp}(x))_{p,q=1}^n$. 
Then $\sum_{j,p,q=1}^nk_jk_pW^{jp}(x)\geq \lambda(x)\sum_{j=1}^nk_j^2$, 
and so we can estimate
\begin{align*}
&\Ideal{\D\xi,\Psi\circ(\D_1\wh{\otimes}1)\wh{\otimes}1\circ\Psi^{-1}\xi}+\Ideal{\Psi\circ(\D_1\wh{\otimes}1)\wh{\otimes}1\circ\Psi^{-1}\xi,\D\xi}\\
&\geq8\pi^2\inf_{x\in M}\{\lambda(x)\}\sum_{j,k}k_j^2\|P_k\xi\|^2-8\pi^2n
\sup_{j,p}\Big\{|\ell_p|\sup_{x\in M}\{|W^{jp}(x)|\}\Big\}\sum_{r,k}|k_r|\|P_k\xi\|^2\\
&\quad-\sum_{p,k}|k_p-\ell_p|C_p\|P_k\xi\|^2\geq\sum_{k\in\Z^n}\bigg(a\sum_{j}k_j^2-b\sum_{j}|k_j|-d\sum_{j}|k_j-\ell_j|\bigg)\|P_k\xi\|^2,
\end{align*}
where we have relabelled some constants and set $d:=\sup_p\{C_p\}$. 
Since $M$ is compact, the constant $a=8\pi^2\inf_{x\in M}\{\lambda(x)\}$ 
is strictly positive, and so the function 
$$
Q:\Z^n\rightarrow\R,\qquad Q(k)=a\sum_{j}k_j^2-b\sum_{j}|k_j|-d\sum_{j}|k_j-\ell_j|
$$
is bounded from below by some $R\in\R$. Hence 
\begin{align*}
\Ideal{\D\xi,\Psi\circ(\D_1\wh{\otimes}1)\wh{\otimes}1\circ\Psi^{-1}\xi}+\Ideal{\Psi\circ(\D_1\wh{\otimes}1)\wh{\otimes}1\circ\Psi^{-1}\xi,\D\xi}\geq R\sum_{k\in\Z^n}\|P_k\xi\|^2=R\|\xi\|^2.&&&\qedhere
\end{align*}
\end{proof}
\subsection{The constructive Kasparov product for manifolds.}
Recall that $(M,g)$ is a compact Riemannian manifold with a 
free, isometric left action by $\mathbb{T}^n$, 
$(S,\nabla^S)$ is an equivariant Clifford module over $M$ with 
Dirac operator $\D$, and $\ell\in\Z^n$ is fixed.

We have seen that  $(C^\infty(M),L^2(S),\D)$ represents the 
product of the unbounded Kasparov modules 
$(\oplus_kC(M)_k,(E_1\wh{\otimes}_{C(M)^{\mathbb{T}^n}\wh{\otimes}\bC}(C(M)^{\mathbb{T}^n}\wh{\otimes}(\Gamma(\bs{\$}_{\mathbb{T}^n})^{\mathbb{T}^n}\wh{\otimes}\bC)^*))_{C(M)^{\mathbb{T}^n}\wh{\otimes}\Gamma(\Cl(\mathbb{T}^n))^{\mathbb{T}^n}},\D_1\wh{\otimes}1)$\,\, (the product of the left-hand and middle modules) and $(C^\infty(M)^{\mathbb{T}^n}\wh{\otimes}\Gamma(\Cl(\mathbb{T}^n))^{\mathbb{T}^n},\h_\ell,\D_\ell)$ (the right-hand module). We now show that the constructive Kasparov product \cites{BMS,KaadLesch,MR} 
can be used to produce a representative of the product of these two cycles. The representative thus obtained is unitarily equivalent 
to\footnote{Here we replace 
the algebra $\oplus_kC(M)_k$ by $\oplus_kC^\infty(M)_k$, and even by $C^\infty(M)$. 
The distinction between these algebras is 
unimportant for $KK$-classes, but may produce differences for (unitary 
equivalence classes of) spectral triples, 
where the choice of smooth algebra enters. We will ignore numerous subtleties involved
in the choice of smooth algebra, which is harmless in the context of 
first order differential operators on compact manifolds.} $(C^\infty(M),L^2(S),T)$ 
for some self-adjoint, first order elliptic differential operator $T$ on $S$. 
If the orbits of $\mathbb{T}^n$ are embedded isometrically into $M$, then $T$ is a bounded 
perturbation of the original operator $\D$.
\begin{defn}
Let $G$ be a compact group, and let $A$ and $B$ be $\Z_2$-graded $C^*$-algebras 
carrying respective actions $\alpha$ and $\beta$ by $G$. 
Let $E_A$ be a $\Z_2$-graded right Hilbert $A$-module with a homomorphism 
$V$ from $G$ into the invertible degree zero bounded operators on $E$ such that 
$V_g(ea)=V_g(e)\alpha_g(a)$ for all $g\in G$, $a\in A$ and $e\in E$, 
and let $(\A,F_B,T)$ be an unbounded equivariant Kasparov $A$-$B$-module. There is 
a natural action of $G$ on $E\wh{\otimes}_A\End_B(F_B)$ given by $g\cdot (e\wh{\otimes}B)=V_g(e)\wh{\otimes}U_gBU_g^{-1}$, 
where $U$ is the action of $G$ on $F_B$. 
A $T$-\emph{connection} on $E_A$ is a linear map $\nabla$ from a dense subspace 
$\E\subset E_A$ which is a right $\A$-module into 
$E\wh{\otimes}_A\End_B(F_B)$, such that $g\cdot\nabla(e)=\nabla(V_g(e))$ for all $g\in G$, $e\in E$, and
\begin{align}\label{eq:z2leibniz}
\nabla(ea)=\nabla(e)a+(-1)^{\deg e}e\wh{\otimes}[T,a]_\pm,\qquad e\in \E,\, a\in\A.
\end{align}
We define a closed operator $1\wh{\otimes}_{\nabla}T$ initially on $\Span\{e\wh{\otimes}f:e\in \E,f\in\dom(T)\}\subset E\wh{\otimes}_AF$ by
$$(1\wh{\otimes}_{\nabla}T)(e\wh{\otimes} f)=(-1)^{\deg e}e\wh{\otimes}Tf+\nabla(e)f.$$
The equivariance of $\nabla$ ensures that $1\wh{\otimes}_\nabla T$ is equivariant. We say that $\nabla$ is \emph{Hermitian} if 
$$(e_1|\nabla e_2)_{\End_B(F_B)}-(\nabla e_1|e_2)_{\End(F_B)}=(-1)^{\deg e_1}[T,(e_1|e_2)_A]_\pm,\quad e_1,\,e_2\in\E.$$
If $\nabla$ is Hermitian, then the operator $1\wh{\otimes}_\nabla T$ is symmetric.
\end{defn}
Let $x\in M$. Choose tangent vectors 
$(v_1,\ldots,v_{m-n})$ spanning $\Span\{X_1(x),\ldots,X_n(x)\}^\perp\subset T_xM$, 
where we recall that $X_j$ is the fundamental vector field associated to 
$\frac{\de}{\de t^j}\in T_e\mathbb{T}^n$. Let 
$(x^1,\ldots,x^n,y^1,\ldots,y^{m-n})$ be the geodesic 
normal coordinates around $x$ corresponding to the 
tangent vectors $(X_1(x),\ldots,X_n(x),v_1,\ldots,v_{m-n})$. 
There is a neighbourhood $U$ of $x$ such that 
$U\cong\pi(U)\times \mathbb{T}^n$ as $\mathbb{T}^n$-spaces, 
where $\pi:M\rightarrow M/\mathbb{T}^n$ is the quotient map, 
so the standard coordinates $(t^1,\ldots,t^n)\in(0,1)^n$ on 
$\mathbb{T}^n$ give us coordinates $(t^1,\ldots,t^n,y^1,\ldots,y^{m-n})$ 
in a neighbourhood of $x$. Since $g(X_j(x),v_p)=0$ and 
$X_j=\frac{\de}{\de t^j}$, it follows from the fact that a 
geodesic is orthogonal to one orbit of $\mathbb{T}^n$ 
if and only if it is orthogonal to every orbit of $\mathbb{T}^n$ 
that it intersects, \cite{Reinhart1959}*{Prop. 2}, that 
$g(\frac{\de}{\de t^j},\frac{\de}{\de y^p})=0$ on the 
coordinate chart for $j=1,\ldots,n$, $p=1,\ldots,m-n$.

Let $(U_i)_{i=1}^N$ be a finite cover of $M$ by such 
coordinate neighbourhoods, and for each 
$k\in\mathbb{Z}^n$, $i=1,\ldots,N$, define 
$\chi_{i,k}\in C^\infty(U_i)$, $\chi_{i,k}(t^1,\ldots,t^n,y^1,\ldots,y^{m-n})
=e^{-2\pi i\sum_{j=1}^nk_jt^j}$. Observe that if $g\in C(M)_k$ 
has support in $U_i$, then $g\chi_{i,k}^{-1}\in C(M)^{\mathbb{T}^n}$. 
Let $(\phi_i)_{i=1}^N$ be an invariant partition of unity 
subordinate to $(U_i)_{i=1}^N$, and for each $i=1,\ldots,N$, 
let $\psi_i\in C^\infty(M)$ be an invariant function with support in $U_i$, such that 
$\psi_i$ is 1 in a neighbourhood of $\supp\phi_i$.

Then for $f\in C(M)$,
$$
\Phi_k(f)=\sum_{i}\phi_i\psi_i\Phi_k(f)=\sum_{i}\phi_i\chi_{i,k}(\Phi_k(f)\psi_i\chi_{i,k}^{-1}).
$$
Let $(x_r)_{r=1}^{2^{\floor{n/2}}}$ be an invariant orthonormal frame for 
$\bs{\$}_{\mathbb{T}^n}$ of homogeneous degree, such that $x_1$ is of even degree 
(in the case $\bs{\$}_{\mathbb{T}^n}$ is $\Z_2$-graded). Then given
$$
\big((f\wh{\otimes}u)\wh{\otimes}(h\wh{\otimes}\ol{w})\big)\wh{\otimes}\xi\in(E_1\wh{\otimes}_{C(M)^{\mathbb{T}^n}\wh{\otimes}\bC}(C(M)^{\mathbb{T}^n}\wh{\otimes}(\Gamma(\bs{\$}_{\mathbb{T}^n})^{\mathbb{T}^n}\wh{\otimes}\bC)^*))\wh{\otimes}_{C(M)^{\mathbb{T}^n}\wh{\otimes}\Gamma(\Cl(\mathbb{T}^n))^{\mathbb{T}^n}}L^2(S)_\ell,$$ 
we may write
\begin{align}\label{eq:f}
&\big((f\wh{\otimes}u)\wh{\otimes}(h\wh{\otimes}\ol{w})\big)\wh{\otimes}\xi=\sum_{k\in\Z^n}\sum_{r=1}^{2^{\floor{n/2}}}\sum_{i=1}^N\\
&\qquad\big(\phi_i\chi_{i,k}\wh{\otimes}(\chi_kx_r\wh{\otimes}1))\wh{\otimes}(1\wh{\otimes}\ol{x_1\wh{\otimes}1})\big)\wh{\otimes}\Phi_k(f)\psi_i\chi_{i,k}^{-1} h\eta\big(\!\,_{\Gamma(\Cl(\mathbb{T}^n))^{\mathbb{T}^n}}(x_1\wh{\otimes}(x_r\wh{\otimes}1|\chi_k^{-1}p_{\chi_{k}^{-1}}u)_{\bC}|w)\big)\xi.\notag
\end{align}
Define a $\D_\ell$-connection on
$E_1\wh{\otimes}_{C(M)^{\mathbb{T}^n}\wh{\otimes}\bC}(C(M)^{\mathbb{T}^n}\wh{\otimes}(\Gamma(\bs{\$}_{\mathbb{T}^n})\wh{\otimes}\bC)^*)$ 
by
\begin{align*}
&\nabla\big((f\wh{\otimes}u)\wh{\otimes}(h\wh{\otimes}\ol{w})\big):=\sum_{k\in\Z^n}\sum_{r=1}^{2^{\floor{n/2}}}\sum_{i=1}^N(-1)^{\deg x_r}\\
&((\phi_i\chi_{i,k}\wh{\otimes}(\chi_kx_r\wh{\otimes}1))\wh{\otimes}(1\wh{\otimes}\ol{x_1\wh{\otimes}1}))\wh{\otimes}\big[\D,\Phi_k(f)\psi_i\chi_{i,k}^{-1}h\eta\big(\!\,_{\Gamma(\Cl(\mathbb{T}^n))^{\mathbb{T}^n}}(x_1\wh{\otimes}(x_r\wh{\otimes}1|\chi_{k}^{-1}p_{\chi_{k}^{-1}}u)_{\bC}|w)\big)\big]_\pm\!.
\end{align*}
That $\nabla$ is equivariant and satisfies Equation \eqref{eq:z2leibniz} follows from Equation \eqref{eq:f}. Since $\nabla$
is built from a frame, \cite{MR}, it is also Hermitian.

Writing $1\wh{\otimes}_\nabla\D_\ell=(1\wh{\otimes}1)\wh{\otimes}_\nabla\D_\ell$ and $\D_1\wh{\otimes}1=(\D_1\wh{\otimes}1)\wh{\otimes}1$ for short, the following result shows that the constructive Kasparov product yields a spectral triple.
\begin{thm}
\label{thm:i-can-add}
For $j=1,\ldots,n$, let $X_j\in\Gamma^\infty(M)$ be the fundamental 
vector field associated to $\frac{\de}{\de t^j}\in T_e\mathbb{T}^n$.
 Let $(h_{jk})_{j,k=1}^n=(g(X_j,X_k))_{j,k=1}^n$, $(h^{jk})=(h_{jk})^{-1}$,
 and let $(W^{jk})_{j,k=1}^n$ be the normalisation functions. Then
$$
\Psi\circ\Big(1\wh{\otimes}_\nabla\D_\ell
+\D_1\wh{\otimes}1\Big)\circ\Psi^{-1}=\D+\sum_{j,r=1}^n(W^{rj}-h^{rj})c(X_r^\flat)\nabla^S_{X_j}+B,
$$
where $B\in\Gamma^\infty(\End(S))$. Thus $\Psi\circ(1\wh{\otimes}_\nabla\D_\ell
+\D_1\wh{\otimes}1)\circ\Psi^{-1}$ is a first order, self-adjoint, equivariant, elliptic differential operator.
Hence
$(C^\infty(M),L^2(S),\Psi\circ(1\wh{\otimes}_\nabla\D_\ell
+\D_1\wh{\otimes}1)\circ\Psi^{-1})$ is an equivariant spectral triple representing the Kasparov product 
(which is also
represented by $(C^\infty(M),L^2(S),\D)$).
\end{thm}
\begin{proof}
Given $\xi\in L^2(S)$, 
$$
\Psi^{-1}(\xi)=\sum_{i=1}^N\sum_{k\in\Z^n}\sum_{r=1}^{2^{\floor{n/2}}}\big(\big(\psi_i\chi_{i,k-\ell}\wh{\otimes}(\chi_{k-\ell}^{-1}x_r\wh{\otimes}1)\big)\wh{\otimes}(1\wh{\otimes}\ol{x_r\wh{\otimes}1})\big)\wh{\otimes}\chi_{i,\ell-k}\phi_iP_k\xi.
$$
Using this we can compute
\begin{align}
\Psi\circ1\wh{\otimes}_\nabla\D_\ell\circ\Psi^{-1}&=\sum_{k\in\Z^n}\sum_{r=1}^{2^{\floor{n/2}}}\sum_{i,j=1}^N(-1)^{\deg x_r}\phi_i\chi_{i,k-\ell}\eta\big(\,\!_{\Gamma(\Cl(\mathbb{T}^n))^{\mathbb{T}^n}}(x_r\wh{\otimes}1|x_1\wh{\otimes}1)\big)\notag\\
&\qquad\times\big[\D,\psi_j\chi_{j,k-\ell}\psi_i\chi_{i,k-\ell}^{-1}\eta\big(\!\,_{\Gamma(\Cl(\mathbb{T}^n))^{\mathbb{T}^n}}(x_1\wh{\otimes}1|x_r\wh{\otimes}1)\big)\big]_\pm\chi_{j,\ell-k}\phi_jP_k\notag\\
&\quad+\sum_{i=1}^N\sum_{k\in\Z^n}\sum_{r=1}^{2^{\floor{n/2}}}\psi_i\chi_{i,k-\ell}\eta\big(\,\!_{\Gamma(\Cl(\mathbb{T}^n))^{\mathbb{T}^n}}(x_r\wh{\otimes}1|x_r\wh{\otimes}1)\big)\D\chi_{i,\ell-k}\phi_iP_k\notag\\
&=\sum_{k\in\Z^n}\sum_{r=1}^{2^{\floor{n/2}}}\sum_{i,j=1}^N(-1)^{\deg x_r}\phi_i\chi_{i,k-\ell}\eta\big(\,\!_{\Gamma(\Cl(\mathbb{T}^n))^{\mathbb{T}^n}}(x_r\wh{\otimes}1|x_1\wh{\otimes}1)\big)\notag\\
&\qquad\times\big[\D,\psi_j\chi_{j,k-\ell}\psi_i\chi_{i,k-\ell}^{-1}\eta\big(\!\,_{\Gamma(\Cl(\mathbb{T}^n))^{\mathbb{T}^n}}(x_1\wh{\otimes}1|x_r\wh{\otimes}1)\big)\big]_\pm\chi_{j,\ell-k}\phi_jP_k\notag\\
&\quad+\sum_{i=1}^N\sum_{k\in\Z^n}\chi_{i,k-\ell}[\D,\psi_i\chi_{i,\ell-k}]\phi_iP_k+\D,\label{eq:firstterm}
\end{align}
where we have used $\sum_{r=1}^{2^{\floor{n/2}}}\,\!_{\Gamma(\Cl(\mathbb{T}^n))^{\mathbb{T}^n}}(x_r\wh{\otimes}1|x_1\wh{\otimes}1)=1$ and $\sum_{i=1}^N\phi_i=1$. Let $I$ denote the first term of Equation \eqref{eq:firstterm}. By several applications of the graded commutator relation $[a,bc]_\pm=(-1)^{\deg b}b[a,c]_\pm+[a,b]_\pm c$, the first term of Equation \eqref{eq:firstterm} can be simplified to
\begin{align*}
I&=\sum_{k\in\Z^n}\sum_{j=1}^N[\D,\psi_j\chi_{j,k-\ell}]\chi_{j,\ell-k}\phi_jP_k+\sum_{k\in\Z^n}\sum_{r=1}^{2^{\floor{n/2}}}\sum_{i=1}^N(-1)^{\deg x_r}\phi_i\chi_{i,k-\ell}\notag\\
&\qquad\quad\qquad\qquad\qquad\times\eta\big(\,\!_{\Gamma(\Cl(\mathbb{T}^n))^{\mathbb{T}^n}}(x_r\wh{\otimes}1|x_1\wh{\otimes}1)\big)[\D,\psi_i\chi_{i,k-\ell}^{-1}]\eta\big(\!\,_{\Gamma(\Cl(\mathbb{T}^n))^{\mathbb{T}^n}}(x_1\wh{\otimes}1|x_r\wh{\otimes}1)\big)P_k\notag\notag\\
&\quad+\sum_{r=1}^{2^{\floor{n/2}}}(-1)^{\deg x_r}\eta\big(\,\!_{\Gamma(\Cl(\mathbb{T}^n))^{\mathbb{T}^n}}(x_r\wh{\otimes}1|x_1\wh{\otimes}1)\big)\big[\D,\eta\big(\!\,_{\Gamma(\Cl(\mathbb{T}^n))^{\mathbb{T}^n}}(x_1\wh{\otimes}1|x_r\wh{\otimes}1)\big)\big]_\pm.\notag
\end{align*}
With respect to the $(t^1,\ldots,t^n,y^1,\ldots,y^{m-n})$ coordinates on $U_i$, $\chi_{i,k}=e^{-2\pi i\sum_{j=1}^nt^jk_j}$, and so
\begin{align*}
\chi_{i,k}^{-1}[\D,\psi_i\chi_{i,k}]=\chi_{i,k}^{-1}c(d\chi_{i,k})=-2\pi i\sum_{j=1}^nk_jc(dt^j).
\end{align*}
Write $\D=\sum_{j=1}^nc(dt^j)\nabla^S_{X_j}+\sum_{s=1}^{m-n}c(dy^s)\nabla^S_{\de_{y^s}}$. Since $g(\de_{t^j},\de_{y^p})=0$ and $X_j^\flat=\sum_{p=1}^nh_{jk}dt^k$, the Clifford vector $c(dy^p)$ anticommutes with $c(X_j^\flat)$ and hence graded commutes with the image of $\Gamma(\Cl(\mathbb{T}^n))^{\mathbb{T}^n}$ under $\eta$ for each $p=1,\ldots,m-n$. Using this fact as well as the compatibility of $\nabla^S$ with the Levi-Civita connection, the first term of Equation \eqref{eq:firstterm} is locally
\begin{align*}
&I=-2\pi i\sum_{k\in\Z^n}\sum_{j=1}^nc(dt^j)(k_j-\ell_j)P_k+2\pi i\sum_{k\in\Z^n}\sum_{r=1}^{2^{\floor{n/2}}}\sum_{j=1}^n(-1)^{\deg x_r}\eta\big(\,\!_{\Gamma(\Cl(\mathbb{T}^n))^{\mathbb{T}^n}}(x_r\wh{\otimes}1|x_1\wh{\otimes}1)\big)\\
&\qquad\qquad\qquad\qquad\qquad\qquad\qquad\qquad\qquad\qquad\times c(dt^j)\eta\big(\!\,_{\Gamma(\Cl(\mathbb{T}^n))^{\mathbb{T}^n}}(x_1\wh{\otimes}1|x_r\wh{\otimes}1)\big)(k_j-\ell_j)P_k\\
&+\sum_{r=1}^{2^{\floor{n/2}}}\sum_{j=1}^n(-1)^{\deg x_r}\eta\big(\,\!_{\Gamma(\Cl(\mathbb{T}^n))^{\mathbb{T}^n}}(x_r\wh{\otimes}1|x_1\wh{\otimes}1)\big)\big[c(dt^j)\nabla^S_{X_j},\eta\big(\!\,_{\Gamma(\Cl(\mathbb{T}^n))^{\mathbb{T}^n}}(x_1\wh{\otimes}1|x_r\wh{\otimes}1)\big)\big]_\pm\notag\\
&+\sum_{r=1}^{2^{\floor{n/2}}}\sum_{p=1}^{m-n}\sum_{j=1}^n(-1)^{\deg x_r}\eta\big(\,\!_{\Gamma(\Cl(\mathbb{T}^n))^{\mathbb{T}^n}}(x_r\wh{\otimes}1|x_1\wh{\otimes}1)\big)\big[c(dy^p)\nabla^S_{\de_{y^p}},\eta\big(\!\,_{\Gamma(\Cl(\mathbb{T}^n))^{\mathbb{T}^n}}(x_1\wh{\otimes}1|x_r\wh{\otimes}1)\big)\big]_\pm\notag\\
&=\sum_{j=1}^nc(dt^j)(\nabla^S_{X_j}+\omega_j-2\pi\ell_j)-\sum_{r=1}^{2^{\floor{n/2}}}\sum_{j=1}^n(-1)^{\deg x_r}\eta\big(\,\!_{\Gamma(\Cl(\mathbb{T}^n))^{\mathbb{T}^n}}(x_r\wh{\otimes}1|x_1\wh{\otimes}1)\big)c(dt^j)\\
&\qquad\qquad\qquad\qquad\qquad\qquad\qquad\qquad\qquad\times\eta\big(\!\,_{\Gamma(\Cl(\mathbb{T}^n))^{\mathbb{T}^n}}(x_1\wh{\otimes}1|x_r\wh{\otimes}1)\big)(\nabla^S_{X_j}+\omega_j-2\pi\ell_j)\\
&\quad+\sum_{r=1}^{2^{\floor{n/2}}}\sum_{j=1}^n(-1)^{\deg x_r}\eta\big(\,\!_{\Gamma(\Cl(\mathbb{T}^n))^{\mathbb{T}^n}}(x_r\wh{\otimes}1|x_1\wh{\otimes}1)\big)c(dt^j)\nabla^{LC}_{X_j}\big(\eta\big(\!\,_{\Gamma(\Cl(\mathbb{T}^n))^{\mathbb{T}^n}}(x_1\wh{\otimes}1|x_r\wh{\otimes}1)\big)\big)\notag\\
&\quad+\sum_{r=1}^{2^{\floor{n/2}}}\sum_{j=1}^n(-1)^{\deg x_r}\eta\big(\,\!_{\Gamma(\Cl(\mathbb{T}^n))^{\mathbb{T}^n}}(x_r\wh{\otimes}1|x_1\wh{\otimes}1)\big)\big[c(dt^j),\eta\big(\!\,_{\Gamma(\Cl(\mathbb{T}^n))^{\mathbb{T}^n}}(x_1\wh{\otimes}1|x_r\wh{\otimes}1)\big)\big]_\pm\nabla^S_{X_j}\notag\\
&\quad+\sum_{r=1}^{2^{\floor{n/2}}}\sum_{p=1}^{m-n}(-1)^{\deg x_r}\eta\big(\,\!_{\Gamma(\Cl(\mathbb{T}^n))^{\mathbb{T}^n}}(x_r\wh{\otimes}1|x_1\wh{\otimes}1)\big)c(dy^p)\nabla^{LC}_{\de_{y^p}}\big(\eta\big(\!\,_{\Gamma(\Cl(\mathbb{T}^n))^{\mathbb{T}^n}}(x_1\wh{\otimes}1|x_r\wh{\otimes}1)\big)\big)\notag
\end{align*}
for $\omega_j\in\Gamma^\infty(\End(S))$ for $j=1,\ldots,n$, using $A_j=2\pi\sum_{k\in\Z^n}k_jP_k$ and Lemma \ref{lem:infgen}. Here $\nabla^{LC}$ denotes the extension of the Levi-Civita connection on the cotangent bundle to the Clifford bundle. Using $\sum_{r=1}^{2^{\floor{n/2}}}\,\!_{\Gamma(\Cl(\mathbb{T}^n))^{\mathbb{T}^n}}(x_r\wh{\otimes}1|x_1\wh{\otimes}1)\,\!_{\Gamma(\Cl(\mathbb{T}^n))^{\mathbb{T}^n}}(x_1\wh{\otimes}1|x_r\wh{\otimes}1)=1$ and the fact that $c(dy^p)$ graded commutes with the image of $\eta$, we can make some cancellations and, working locally, simplify the first term of Equation \eqref{eq:firstterm} to
\begin{align*}
I&=\sum_{j=1}^nc(dt^j)(\omega_j-2\pi\ell_j)\\
&-\sum_{r=1}^{2^{\floor{n/2}}}\sum_{j=1}^n(-1)^{\deg x_r}\eta\big(\,\!_{\Gamma(\Cl(\mathbb{T}^n))^{\mathbb{T}^n}}(x_r\wh{\otimes}1|x_1\wh{\otimes}1)\big)c(dt^j)\eta\big(\!\,_{\Gamma(\Cl(\mathbb{T}^n))^{\mathbb{T}^n}}(x_1\wh{\otimes}1|x_r\wh{\otimes}1)\big)(\omega_j-2\pi\ell_j)\notag\\
&+\sum_{r=1}^{2^{\floor{n/2}}}\sum_{j=1}^n(-1)^{\deg x_r}\eta\big(\,\!_{\Gamma(\Cl(\mathbb{T}^n))^{\mathbb{T}^n}}(x_r\wh{\otimes}1|x_1\wh{\otimes}1)\big)c(dt^j)\nabla^{LC}_{X_j}\big(\eta\big(\!\,_{\Gamma(\Cl(\mathbb{T}^n))^{\mathbb{T}^n}}(x_1\wh{\otimes}1|x_r\wh{\otimes}1)\big)\big)\notag\\
&+\sum_{r=1}^{2^{\floor{n/2}}}\sum_{p=1}^{m-n}(-1)^{\deg x_r}\eta\big(\,\!_{\Gamma(\Cl(\mathbb{T}^n))^{\mathbb{T}^n}}(x_r\wh{\otimes}1|x_1\wh{\otimes}1)\big)c(dy^p)\nabla^{LC}_{\de_{y^p}}\big(\eta\big(\!\,_{\Gamma(\Cl(\mathbb{T}^n))^{\mathbb{T}^n}}(x_1\wh{\otimes}1|x_r\wh{\otimes}1)\big)\big)\\
&\in\Gamma^\infty(\End(S)).
\end{align*}
The second term of Equation \eqref{eq:firstterm} is 
\begin{align*}
&\sum_{i=1}^N\sum_{k\in\Z^n}\chi_{i,k-\ell}[\D,\psi_i\chi_{i,\ell-k}]\phi_iP_k=2\pi i\sum_{k\in\Z^n}\sum_{j=1}^nc(dt^j)(k_j-\ell_j)P_k\\
&=-\sum_{j=1}^nc(dt^j)(\nabla^S_{X_j}+\omega_j-2\pi\ell_j)=-\sum_{j,q=1}^nh^{jq}c(X_q^\flat)(\nabla^S_{X_j}+\omega_j-2\pi\ell_j)
\end{align*}
for some $\omega_j\in\Gamma^\infty(\End(S))$ by Lemma \ref{lem:infgen}. Putting the expressions for Equation \eqref{eq:firstterm} together with Lemma \ref{lem:d1otimes1} and Lemma \ref{lem:infgen} yields
\begin{align*}
\Psi\circ(1\wh{\otimes}_\nabla\D_\ell
+\D_1\wh{\otimes}1)\circ\Psi^{-1}&=\D+\sum_{j,r=1}^n(W^{rj}-h^{rj})c(X_r^\flat)\nabla^S_{X_j}+B\\
&=\sum_{p=1}^{m-n}c(dy^p)\nabla^S_{\de_{y^p}}+\sum_{j,r,q=1}^nW^{rj}h_{rq}c(dt^q)\nabla^S_{X_j}+B
\end{align*}
for some $B\in\Gamma^\infty(\End(S))$, which establishes that 
$\Psi\circ(1\wh{\otimes}_\nabla\D_\ell
+\D_1\wh{\otimes}1)\circ\Psi^{-1}$ is a first order differential operator. Since 
$(W^{rj})_{rj=1}^n$ and $(h_{rq})_{r,q=1}^n$ are invertible, this also shows 
that the operator $\Psi\circ(1\wh{\otimes}_\nabla\D_\ell
+\D_1\wh{\otimes}1)\circ\Psi^{-1}$ is elliptic. 
Since $\nabla$ is Hermitian, 
$1\wh{\otimes}_\nabla\D_\ell$ is symmetric, and so 
$1\wh{\otimes}_\nabla\D_\ell+\D_1\wh{\otimes}1$ 
is the sum of a symmetric 
operator with a self-adjoint operator, which is symmetric. 
Elliptic operator 
theory, \cites{HigsonRoe,LawsonMichelsohn}, implies that $\Psi\circ(1\wh{\otimes}_\nabla\D_\ell
+\D_1\wh{\otimes}1)\circ\Psi^{-1}$ is essentially self-adjoint 
with compact resolvent, and hence 
$(C^\infty(M),L^2(S),\Psi\circ(1\wh{\otimes}_\nabla\D_\ell
+\D_1\wh{\otimes}1)\circ\Psi^{-1})$ is an equivariant spectral triple. That 
$(C^\infty(M),L^2(S),\Psi\circ(1\wh{\otimes}_\nabla\D_\ell
+\D_1\wh{\otimes}1)\circ\Psi^{-1})$ represents the product is now a straightforward application
of Kucerovsky's criteria.
\end{proof}
\begin{cor}
\label{cor:connectionfactorisation}
Suppose that each orbit is an isometric embedding of 
$\mathbb{T}^n$ in $M$. That is, the fundamental vector fields  
$T_e\mathbb{T}^n\ni v\mapsto X^{(v)}\in\Gamma^\infty(TM)$ 
satisfy $(X^{(v)}|X^{(v)})_{C(M)}=\|v\|^2$. Then
$$
\D-\Psi\circ(1\wh{\otimes}_\nabla\D_\ell
+\D_1\wh{\otimes}1)\circ\Psi^{-1}\in \Gamma^\infty(\End(S)).
$$
\end{cor}
\begin{proof}
In this case, the normalisation functions are $W^{jk}=\delta^{jk}$, 
and so  
Lemma \ref{thm:i-can-add} becomes $\Psi\circ(1\wh{\otimes}_\nabla\D_\ell
+\D_1\wh{\otimes}1)\circ\Psi^{-1}=\D+B$ where $B\in\Gamma^\infty(\End(S))$.
\end{proof}
\section{Example: the Dirac operator on the 2-sphere.}
\label{sec:dirac-ess2}
The spinor Dirac operator $\D$ on $S^2$ defines an even spectral triple 
$(C^\infty(S^2),L^2(\bs{\$}_{S^2}),\D)$. The circle acts on $S^2$ by 
rotation about the north-south axis, and there are countably infinitely 
many lifts of this action to $L^2(\bs{\$}_{S^2})$, such that 
$(C^\infty(S^2),L^2(\bs{\$}_{S^2}),\D)$ is an equivariant spectral triple. 
One can then ask whether any of these spectral triples can be factorised, but since the 
action of $\mathbb{T}$ on $S^2$ is not free we cannot apply the earlier theory.

In fact, we cannot factorise $(C^\infty(S^2),L^2(\bs{\$}_{S^2}),\D)$, 
since the spectral subspace assumption is not satisfied, and, more seriously, 
$K^1(C(S^2)^{\mathbb{T}})= K^1([0,1])=\{0\}$. Since the class of the triple
$(C^\infty(S^2),L^2(\bs{\$}_{S^2}),\D)$ in $K^0(C(S^2))$ is non-zero, 
it is impossible to recover this class under the Kasparov product 
between $KK^1(C(S^2),C(S^2)^{\mathbb{T}})$ and $KK^1(C(S^2)^{\mathbb{T}},\C)=\{0\}$.

Instead, we remove the poles, and restrict the spectral triple to $(C^\infty_c(S^2\take\{N,S\},L^2(\bs{\$}_{S^2}),\D)$ and ask whether this equivariant spectral triple can be factorised. The circle now acts freely, and hence the spectral subspace assumption is satisfied.

We show that factorisation is achieved for 
$(C^\infty_c(S^2\take\{N,S\}),L^2(\bs{\$}_{S^2}),\D)$ 
for every possible lift of the circle action. Unlike for a 
free action on a compact manifold, the positivity criterion 
is satisfied for precisely two choices of the character 
$\ell\in\Z$ of Defintion \ref{defn:theta} used to define the right-hand module.

We will describe the Dirac operator $\D$ on the spinor bundle 
$\bs{\$}_{S^2}$ over $S^2$, \cites{VarillyIntro,Gracia}.

Let $N$ be the North pole of $S^2$, and let $U_N$ be $S^2\take\{N\}$. 
A chart for $U_N$ is given by stereographic projection onto $\C$. 
This chart defines a trivialisation of the spinor bundle $\bs{\$}_{S^2}$. 
All work will be done in the $U_N$ trivialisation unless explicitly stated 
otherwise. We will work in the standard polar coordinates 
$(\theta,\phi)\in(0,\pi)\times(0,2\pi)$.

The spinor Dirac operator 
is given by
\begin{align}\label{eq:diracpolar}
\D&=\begin{pmatrix}0&e^{i\phi}\big(i\de_\theta+\csc(\theta)\de_\phi+i\cot(\theta/2)/2\big)\\e^{-i\phi}\big(i\de_\theta-\csc(\theta)\de_\phi+i\cot(\theta/2)/2\big)&0\end{pmatrix}.
\end{align}
The Hilbert space $L^2(\bs{\$}_{S^2})$ is graded by $\gamma=\left(\begin{smallmatrix}1&0\\0&-1\end{smallmatrix}\right)$. 
The action of the circle $\mathbb{T}$ on $S^2$ is $t\cdot(\theta,\phi)=(\theta,\phi+2\pi t)$. 
There are countably infinitely many lifts of this action which make 
$(C^\infty(S^2),L^2(\bs{\$}_{S^2}),\D)$ into a $\mathbb{T}$-equivariant spectral triple.
\begin{prop}\label{prop:tk}
Any even unitary action of $\mathbb{T}$ on $L^2(\bs{\$}_{S^2})$ 
which commutes with $\D$ and which is compatible with the action on 
$C(S^2)$ is equal to $V_k:\mathbb{T}\rightarrow U(L^2(\bs{\$}_{S^2}))$ 
for some $k\in\Z$, where
$$
V_{k,t}\begin{pmatrix}f(\theta,\phi)\\g(\theta,\phi)\end{pmatrix}
:=\begin{pmatrix}e^{2\pi ikt}f(\theta,\phi-2\pi t)\\e^{2\pi i(k-1)t}g(\theta,\phi-2\pi t)\end{pmatrix}.
$$
\end{prop}
\begin{proof}
We require the action of $\mathbb{T}$ on $L^2(\bs{\$}_{S^2})$ to be 
compatible with the action $\alpha$ of $\mathbb{T}$ on $C(S^2)$, 
which is $\alpha_t(f)(\theta,\phi) =f(\theta,\phi-2\pi t)$. Hence the action 
on spinors is of the form
$$
V_t\begin{pmatrix}f(\theta,\phi)\\g(\theta,\phi)\end{pmatrix}
=\begin{pmatrix}a&b\\d&h\end{pmatrix}\begin{pmatrix}f(\theta,\phi-2\pi t)\\g(\theta,\phi-2\pi t)\end{pmatrix},
$$
where $a,b,d$ and $h$ can \emph{a priori} depend on $\theta$, $\phi$ 
and $t$. Since the action of $\mathbb{T}$ should commute with the 
grading, we require $b=d=0$. Requiring that the action is unitary, that it 
commutes with $\D$ and that it is a group homomorphism determines 
that $a=e^{2\pi ikt}$ and $h=e^{2\pi i(k-1)t}$ for some $k\in\Z$.
\end{proof}
\begin{rem}
None of these actions preserve the real structure on $\bs{\$}_{S^2}$, 
so they are $\spin^c$ but not $\spin$ actions. There is however a 
unique lift of the ``double'' action of $\mathbb{T}$, 
$t\cdot (\theta,\phi)=(\theta,\phi+4\pi t)$, 
to a spin action given by setting $k=1/2$ and replacing $t$ by $2t$ in Proposition \ref{prop:tk}. 
\end{rem}
We fix $k\in\Z$ for the remainder of the section, 
fixing a representation $V_k:\mathbb{T}\rightarrow U(L^2(\bs{\$}_{S^2}))$.
The spectral subspaces of $C(S^2)$ are 
\begin{align*}
C(S^2)_j=\left\{\begin{array}{ll}\{f(\theta):f\in C([0,1])\}&\text{if }j=0\\
\{f(\theta)e^{-ij\phi}:f\in C_0((0,1))\}&\text{if }j\neq0.\end{array}\right.
\end{align*}
Hence 
$$
\ol{C(S^2)_jC(S^2)_j^*}\cong\left\{\begin{array}{ll}C([0,1])&\text{if }j=0\\
C_0((0,1))&\text{if }j\neq0.\end{array}\right.
$$
Since $C_0((0,1))$ is not a complemented ideal in 
$C(S^2)^{\mathbb{T}}\cong C([0,1])$, $C(S^2)$ 
does not satisfy the spectral subspace assumption, 
and so we cannot define the left-hand module if we 
use the $C^*$-algebra $C(S^2)$. However, the SSA is 
satisfied for $C_0(S^2\take\{N,S\})$, since the action on $S^2\take\{N,S\}$ is free, by Proposition \ref{prop:fullSS}.

By taking the fundamental vector field map and normalising as in 
Section \ref{sect:manifoldfact}, we define the map 
$\eta:\Gamma(\Cl(\mathbb{T}))^{\mathbb{T}}\rightarrow B(L^2(\bs{\$}_{S^2}))$ by
\begin{align*}
\eta(c(dt))=-\frac{1}{\sqrt{g(d\phi,d\phi)}}c(d\phi)
=\begin{pmatrix}0&-e^{i\phi}\\e^{-i\phi}&0\end{pmatrix}
\end{align*}
We check that $\eta$ satisfies the conditions 
of Definition \ref{defn:theta}. Clearly $\eta(c(dt))$ commutes with the 
algebra, so Condition (1) is satisfied. Since $a\eta(c(dt))$ is a smooth 
bundle endomorphism for all $a\in C^\infty_c(S^2\take\{S,N\})$, 
$a\eta(c(dt))$ preserves $\dom(\D)$. It remains to check the commutation condition. 
We compute:
\begin{align*}
&\{\D,\eta(c(dt))\}=\\
&\bigg\{\begin{pmatrix}0&e^{i\phi}\big(i\de_\theta+\csc(\theta)\de_\phi+i\cot(\theta/2)/2\big)\\
e^{-i\phi}\big(i\de_\theta-\csc(\theta)\de_\phi+i\cot(\theta/2)/2\big)&0\end{pmatrix},\begin{pmatrix}0&-e^{i\phi}\\e^{-i\phi}&0\end{pmatrix}\bigg\}\\
&=\begin{pmatrix}2\csc(\theta)\de_\phi-i\csc(\theta)&0\\0&2\csc(\theta)\de_\phi+i\csc(\theta)\end{pmatrix}.
\end{align*}
Hence if $f(\theta)e^{- ij\phi}\in C^\infty_c(S^2\take\{S,N\})_j$, then 
\begin{align*}
&\{\D,\eta(c(dt))\}f(\theta)e^{-ij\phi}P_\ell\\
&=\begin{pmatrix}2i\csc(\theta)(k-\ell-j)-i\csc(\theta)&0\\0&2i\csc(\theta)(k-\ell-j-1)+i\csc(\theta)\end{pmatrix}f(\theta)e^{-ij\phi}P_\ell\\
&=-i\csc(\theta)(2j+2\ell-2k+1)f(\theta)e^{-ij\phi}P_\ell.
\end{align*}
Since $f\in C_c((0,\pi))$, this is bounded, and so 
Condition (2) of Definition \ref{defn:theta} is satisfied. 
Therefore $(\ell,\eta)$ satisfy the conditions of Definition \ref{defn:theta} for any $\ell\in\Z$.

Let $n,\ell\in\Z$, and let 
$\xi=\left(\begin{smallmatrix}f(\theta)e^{i(k-n-\ell)\phi}\\g(\theta)e^{i(k-n-\ell-1)\phi}\end{smallmatrix}\right)
\in\dom(\D)\cap L^2(\bs{\$}_{S^2})_{n+\ell}$. 
Then the positivity criterion reduces to
\begin{align*}
&\Ideal{\D\xi,in\eta(c(dt))P_{n+\ell}\xi}+\Ideal{in\eta(c(dt))P_{n+\ell}\xi,\D\xi}\\
&=\int_0^\pi\!\!\!\int_0^{2\pi}d\phi\,d\theta\sin(\theta)\times\\
&\Big(\ol{e^{i(k-n-\ell)\phi}(ig'(\theta)+i(k-n-\ell-1)\csc(\theta)g(\theta)+i\cot(\theta/2)g(\theta)/2)}(-ine^{i(k-n-\ell)\phi}g(\theta))\\
&\quad+\ol{e^{i(k-n-\ell-1)\phi}(if'(\theta)-i(k-n-\ell)\csc(\theta)f(\theta)+i\cot(\theta/2)f(\theta)/2)}(ine^{i(k-n-\ell-1)\phi}f(\theta))\\
&\quad+\ol{-ine^{i(k-n-\ell)\phi}g(\theta)}e^{i(k-n-\ell)\phi}(ig'(\theta)+i(k-n-\ell-1)\csc(\theta)g(\theta)+i\cot(\theta/2)g(\theta)/2)\\
&\quad+\ol{ine^{i(k-n-\ell-1)\phi}f(\theta)}e^{i(k-n-\ell-1)\phi}(if'(\theta)-i(k-n-\ell)\csc(\theta)f(\theta)+i\cot(\theta/2)f(\theta)/2)\Big)\\
&=4\pi n(n-k+\ell+1/2)\int_0^\pi d\theta\Big(|f(\theta)|^2+|g(\theta)|^2\Big).
\end{align*}
If $p(n)=2n(n-k+\ell+1/2)$ is non-negative for all $n\in\Z$, 
then the factorisation condition is satisfied. Conversely, since 
$\int_0^\pi d\theta\Big(|f(\theta)|^2+|g(\theta)|^2\Big)$ is not 
bounded by $\|\xi\|^2$, if $p(n)<0$ for some $n\in\Z$, then 
$\Ideal{\D\xi,-in\eta(c(dt))P_{n+\ell}\xi}+\Ideal{-in\eta(c(dt))P_{n+\ell}\xi,\D\xi}$ 
is not bounded from below and the factorisation condition is not satisfied.

Since $\ell\in\Z$ has thus far not been fixed, we will 
determine for which values of $\ell$ the polynomial 
$p:\Z\rightarrow\R$ is non-negative. As a real-valued polynomial, 
$p$ has a minimum at $x=(k-\ell)/2-1/4$.

Suppose $k-\ell$ is even. Then the integer values of 
$n$ either side of this minimum are $n=(k-\ell)/2-1$ and 
$n=(k-\ell)/2$, at which $p(n)$ has respective values 
$-(\ell-k+2)(\ell-k-1)/2$ and $-(\ell-k+1)(\ell-k)/2$. 
The smallest of these two values is $p((k-\ell)/2)=-(\ell-k+1)(\ell-k)/2$. 
As a function of $\ell$, $q(\ell)=-(\ell-k+1)(\ell-k)/2$ has a 
maximum at $\ell=k-1/2$. The integer values on either side 
of this with $k-\ell$ even are $\ell=k$ and $\ell=k-2$, at which 
$q(\ell)$ has respective values 0 and $-1$. Therefore if $k-\ell$ 
is even, then $p(n)$ is non-negative if and only if $\ell=k$.

Suppose now that $k-\ell$ is odd. Then the integer values of 
$n$ either side of the minimum $n=(k-\ell)/2-1/4$ are 
$n=(k-\ell)/2-1/2$ and $n=(k-\ell)/2+1/2$, at which $p(n)$ 
has respective values $-(\ell-k+1)(\ell-k)/2$ and $-(\ell-k+2)(\ell-k-1)/2$, 
the smallest of which is $p((k-\ell)/2-1/2)=-(\ell-k+1)(\ell-k)/2$. 
As a function of $\ell$, $r(\ell)=-(\ell-k+1)(\ell-k)/2$ has a 
maximum at $\ell=k-1/2$. The values on either side such that 
$k-\ell$ is odd are $\ell=k-1$ and $\ell=k+1$, at which $r(\ell)$ 
has respective values 0 and $-1$. Therefore if $k-\ell$ is odd, 
then $p(n)$ is non-negative if and only if $\ell=k-1$.

Thus factorisation is achieved for the equivariant spectral triple 
$(C^\infty_c(S^2\take\{N,S\}),L^2(\bs{\$}_{S^2}),\D)$ for any lift $V_k$ of the circle 
action to $L^2(\bs{\$}_{S^2})$, by choosing the 
characters $\ell=k$ or $\ell=k-1$ when constructing the right-hand module.

We conclude the 2-sphere example by examining the operator 
on the right-hand module, which, upon identifying 
$C_0(S^2\take\{N,S\})^{\mathbb{T}}$ with $C_0((0,\pi))$ and 
$\Gamma(\Cl(\mathbb{T}))^{\mathbb{T}}$ with $\Cl_1$, defines a 
spectral triple for $C_0((0,\pi))\wh{\otimes}\Cl_1$ . One might wonder 
whether it can be obtained from an odd spectral triple for $C_0((0,\pi))$, 
such as that defined by (some self-adjoint extension of) the Dirac 
operator on $(0,\pi)$. We show that this is not the case; for each 
$\ell\in\Z$ there is no odd spectral triple $(C^\infty_c((0,\pi)),\h',\D')$ 
such that the right-hand module is the even spectral triple 
corresponding to $(C^\infty_c((0,\pi)),\h',\D')$.

Let $k,\ell\in\Z$ be fixed, where $V_k:\mathbb{T}\rightarrow U(L^2(\bs{\$}_{S^2}))$ 
is the representation and $(\ell,\eta)$ is the pair of Definition \ref{defn:theta}. 
Define $F:\h_\ell\rightarrow L^2([0,\pi])\wh{\otimes}\C^2$ by
$$
F\left(\begin{pmatrix}f(\theta)e^{i(k-\ell)\phi}\\g(\theta)e^{i(k-\ell-1)\phi}\end{pmatrix}\right)=\sqrt{\sin\theta}\begin{pmatrix}if(\theta)\\g(\theta)\end{pmatrix}.
$$
The map $F$ is a $C_0((0,\pi))\wh{\otimes}\Cl_1$-linear $\Z_2$-graded unitary 
isomorphism between $L^2(\bs{\$}_{S^2})_\ell$ and $L^2([0,\pi])\wh{\otimes}\C^2$, 
where the latter space is graded by 
$1\wh{\otimes}\left(\begin{smallmatrix}1&0\\0&-1\end{smallmatrix}\right)$ 
and the action of $\Cl_1$ is given by 
$c\mapsto 1\wh{\otimes}\left(\begin{smallmatrix}0&1\\1&0\end{smallmatrix}\right)$. 
We can compute
\begin{align*}
F\circ \D_\ell\circ F^{-1}=-i\de_\theta\wh{\otimes}\omega-(k-\ell-1/2)\csc(\theta)\wh{\otimes}c,
\end{align*}
where $\omega=\left(\begin{smallmatrix}0&-i\\i&0\end{smallmatrix}\right)$. 
Hence the right-hand module is unitarily equivalent to the spectral triple
$$
\Big(C^\infty_c((0,\pi))\wh{\otimes}\Cl_1,L^2([0,\pi])\wh{\otimes}\C^2,\,-i\de_\theta\wh{\otimes}\omega
-(k-\ell-1/2)\csc(\theta)\wh{\otimes}c\Big).
$$
If $(C^\infty_c((0,\pi)),L^2([0,\pi]),\D')$ is an odd spectral triple, 
then the corresponding even spectral triple is 
$(C^\infty_c((0,\pi))\wh{\otimes}\Cl_1,L^2([0,\pi])\wh{\otimes}\C^2,\D'\wh{\otimes}\omega)$. 
The presence of the $(k-\ell-1/2)\csc(\theta)\wh{\otimes}c$ factor 
means that the right-hand module is not the even spectral triple 
corresponding to any odd spectral triple.


\begin{bibdiv}
\begin{biblist}
\bib{BaajJulg}{article}{title={Th\'{e}orie bivariante de Kasparov et op\'{e}rateurs non born\'{e}es dans les $C^*$-modules hilbertiens},author={S. Baaj},author={P. Julg},journal={C. R. Acad. Sci. Paris},volume={296},date={1983},pages={875--878}}
\bib{BerlineGetzlerVergne}{book}{title={Heat Kernels and Dirac Operators},author={Berline, N.},author={Getzler, E.},author={Vergne, M.},date={2004},publisher={Springer},address={Berlin}}
\bib{BMS}{article}{title={Gauge theory for spectral triples and the unbounded Kasparov product},author={S. Brain},author={B. Mesland},author={W. D. van Suijlekom},date={2013},note={To appear in J. Noncommut. Geom., preprint at arXiv:1306.1951}}
\bib{CGRS2}{article}{title={Index theory for locally compact noncommutative geometries},author={Carey, A. L.},author={Gayral, V.},author={Rennie, A.},author={Sukochev, F. A.},journal={Mem. Amer. Math. Soc.},volume={231},number={1085},date={2014}}
\bib{CNNR}{article}{title={Twisted cyclic theory, equivariant $KK$-theory and KMS states},author={A. L. Carey},author={S. Neshveyev},author={R. Nest},author={A. Rennie},journal={J. reine angew. Math.},volume={650},date={2011},pages={161--191}}
\bib{Connes}{book}{title={Noncommutative Geometry},author={Connes, A.},date={1994},publisher={Academic Press},address={London and San Diego}}
\bib{ConnesLandi}{article}{title={Noncommutative manifolds: the instanton algebra and isospectral deformations},author={Connes, A.},author={Landi, G.},journal={Commun. Math. Phys.},volume={221},date={2001},pages={141--159}}
\bib{DS2013}{article}{title={Noncommutative circle bundles and new Dirac operators},author={L. Dabrowski},author={A. Sitarz},journal={Commun. Math. Phys.},volume={318},date={2013},pages={111--130}}
\bib{DSZ2014}{article}{title={Dirac operators on noncommutative principal circle bundles},author={Dabrowski, L.},author={Sitarz, A.},author={Zucca, A.},journal={Int. J. Geom. Methods Mod. Phys.},volume={11},date={2014}}
\bib{Gracia}{book}{title={Elements of Noncommutative Geometry},author={Gracia-Bond\'{i}a, J. M.},author={V\'{a}rilly, J. C.},author={Figueroa, H.},date={2001},publisher={Birkh\"{a}user},address={Boston}}
\bib{HigsonRoe}{book}{title={Analytic K-Homology},author={Higson, Nigel},author={Roe, John},date={2000},publisher={Oxford University Press},address={Oxford}}
\bib{KaadLesch}{article}{title={Spectral flow and the unbounded Kasparov product},author={Kaad, J.},author={Lesch, M.},journal={Adv. Math.},volume={248},date={2013},pages={495--530}}
\bib{KasparovKK}{article}{title={The operator $K$-functor and extensions of $C^*$-algebras},author={Kasparov, G. G.},journal={Math. USSR Izv.},date={1981},volume={16},pages={513--572}}
\bib{Kucerovsky}{article}{title={The $KK$-product of unbounded modules},author={Kucerovsky, D.},journal={J. $K$-Theory},volume={11},date={1997},pages={17--34}}
\bib{Kucerovsky2000}{article}{title={A lifting theorem giving an isomorphism of $KK$-products in bounded and unbounded $KK$-theory},author={Kucerovsky, D.},journal={J. Operator Theory},volume={44},date={2000},pages={255--275}}
\bib{Lance}{book}{title={Hilbert $C^*$-Modules},author={Lance, E. C.},date={1995},publisher={CUP},address={Cambridge}}
\bib{LawsonMichelsohn}{book}{title={Spin Geometry},author={Lawson, H. B.},author={Michelsohn, M.-L.},date={1989},publisher={PUP},address={Princeton}}
\bib{Mesland}{article}{title={Unbounded bivariant $K$-theory and correspondences in noncommutative geometry},author={Mesland, B.},journal={J. reine angew. Math.},volume={691},date={2014},pages={101--172}}
\bib{MR}{article}{title={Nonunital spectral triples and metric completeness in unbounded $KK$-theory},author={Mesland, B.},author={Rennie, A.},date={2015},note={Preprint, arXiv:1502.04520}}
\bib{ReedSimon1}{book}{title={Methods of Modern Mathematical Physics. I. Functional Analysis. 2nd ed.},author={Reed, M.},author={Simon, B.},date={1980},publisher={Academic Press},address={New York}}
\bib{PR}{article}{title={The noncommutative geometry of graph $C^*$-algebras I: The Index Theorem},author={Pask, D.},author={Rennie, A.},journal={J. Funct. Anal.},volume={233},date={2006},pages={92--134}}
\bib{PRS}{article}{title={The noncommutative geometry of $k$-graph $C^*$-algebras},author={Pask, D.},author={Rennie, A.},author={Sims, A.},journal={J. K-Theory},volume={1},date={2008},pages={259--304}}
\bib{RW}{book}{title={Morita Equivalence and Continuous-Trace $C^*$-algebras},author={Raeburn, I.},author={Williams, D. P.},date={1998},publisher={AMS},address={Providence}}
\bib{Reinhart1959}{article}{title={Foliated manifolds with bundle-like metrics},author={Reinhart, B. L.},journal={Ann. Math.},volume={69},date={1959},pages={119--132}}
\bib{Rieffel}{article}{title={Deformation Quantization for Actions of $\mathbb{R}^d$},author={Rieffel, M. A.},date={1993},journal={Mem. Amer. Math. Soc.},number={506},volume={106}}
\bib{Slebarski1985}{article}{title={Dirac operators on a compact Lie group},author={Slebarski, S.},journal={Bull. London Math. Soc.},volume={17},date={1985},pages={579--583}}
\bib{VarillyIntro}{book}{title={An Introduction to Noncommutative Geometry},author={V\'{a}rilly, J. C.},date={2006},publisher={J. Eur. Math. Soc.},address={Z\"{u}rich}}
\bib{Zucca}{book}{title={Dirac Operators on Quantum Principal $G$-Bundles}, author={A. Zucca}, date={2013}, publisher={SISSA, Digital Library}}
\end{biblist}
\end{bibdiv}
\end{document}